\documentclass[12pt]{amsart}
\usepackage{amssymb,amsmath,amsthm}
\numberwithin{equation}{section}
\usepackage{fullpage}
\usepackage{enumerate}
\usepackage{hyperref}

\newtheorem{thm}{Theorem}[section]
\newtheorem{lem}[thm]{Lemma}

\newtheorem{cor}[thm]{Corollary}
\newtheorem{prop}[thm]{Proposition}
\theoremstyle{definition}

\newtheorem{defn}[thm]{Definition}

\newtheorem{ques}[thm]{Question}

\def\zlstar{{z{\bf L}^*}}

\def\Fz{(I-\zlstar)^{-1}}
\def\Fzplus{(I+\zlstar)}
\def\Fzminus{(I-\zlstar)}

\def\Fw{(I-{{\bf L}w^*})^{-1}}
\def\Fwplus{(I+{{\bf L}w^*})}
\def\Fwminus{(I-{{\bf L}w^*})}

\def\L{{\bf L}}
\def\Fd{\mathbb F_d^+}

\def\nn{\mathbf n}
\def\mm{\mathbf m}
\def\pp{\mathbf p}

\def\abar{{\alpha^*}}
\begin{document}
\title{Clark theory in the Drury-Arveson space}
\author{Michael T. Jury}
\address{Department of Mathematics,
        University of Florida, 
        Gainesville, Florida 32611-8105}
\email{mjury@math.ufl.edu}
\thanks{Partially supported by NSF grant DMS-1101461}
\date{\today}

\begin{abstract}We extend the basic elements of Clark's theory of
  rank-one perturbations of backward shifts, to row-contractive
  operators associated to de Branges-Rovnyak type spaces $\mathcal
  H(b)$ contractively contained in the Drury-Arveson space on the unit
  ball in $\mathbb C^d$. The Aleksandrov-Clark
  measures on the circle are replaced by a family of states on a certain
  noncommutative operator system, and the backward shift is replaced
  by a canonical solution to the Gleason problem in $\mathcal
  H(b)$. In addition we introduce the notion of a ``quasi-extreme''
  multiplier of the Drury-Arveson space and use it to characterize
  those $\mathcal H(b)$ spaces that are invariant under multiplication
  by the coordinate functions.
\end{abstract}
\maketitle

\section{Introduction} The purpose of this paper is to provide one
method of extending the elementary portions of Clark's theory of
rank-one perturbations of backward shifts, to the Drury-Arveson space
$H^2_d$ of the unit ball $\mathbb B^d \subset\mathbb C^d$. This is the
space of functions holomorphic in $\mathbb B^d$ with reproducing
kernel
\begin{equation}
k(z,w)=  \frac{1}{1-zw^*}.
\end{equation}
(Here $zw^*=\sum_{j=1}^d z_jw_j^*$ is the standard Hermitian inner
product in $\mathbb C^d$.) When $d=1$, this is of course the usual
Hardy space $H^2$ in the unit disk. When $d>1$, the space $H^2_d$ is
an analytic Besov space, but is in many ways a more appropriate
higher-dimensional analog of $H^2$ than the classical Hardy space in
the ball (which has the kernel $s(z,w)=(1-{zw^*})^{-d}$). The recent
survey \cite{Sha} provides an overview.

To begin with we explain what is meant by the ``elementary portions of
Clark's theory;'' our treatment is heavily influenced by the
exposition of Sarason \cite{Sar-book} and the treatment of
Aleksandrov-Clark measures in \cite[Chapter 9]{CimMatRos}.  In
particular we take a point of view in which the de Branges-Rovnyak
spaces are central.  Let $b$ be a non-constant function analytic in
the unit disk $\mathbb D\subset \mathbb C$ and bounded by $1$
there. (In Clark's original treatment \cite{Cla} $b$ was assumed to be
an inner function; that is, $|b|=1$ almost everywhere on the unit
circle.) For this discussion we impose the simplifying normalization
$b(0)=0$. Associated to $b$ is a reproducing kernel Hilbert space
$\mathcal H(b)$, with kernel
\begin{equation}
  k^b(z,w)=\frac{1-b(z)b(w)^*}{1-zw^*}
\end{equation}
where $*$ denotes the complex conjugate. (In the inner case, $\mathcal
H(b)$ is isometrically the orthogonal complement of the Beurling
subspace $bH^2$.)  We call $\mathcal H(b)$ the {\em deBranges-Rovnyak
  space} associated to $b$; as a set it is contained in $H^2$ and for
each $f\in \mathcal H(b)$ we have $\|f\|_{H^2}\leq \|f\|_{\mathcal
  H(b)}$, so that $\mathcal H(b)$ is contractively contained in
$H^2$. The basic theory of $\mathcal H(b)$ spaces may be found in the
original book of de Branges and Rovnyak \cite{dBR-book}, and the
monograph of Sarason \cite{Sar-book}.

An important feature of the $\mathcal H(b)$ spaces is that they are
invariant under the backward shift operator
\begin{equation}
  S^*f(z) := \frac{f(z)-f(0)}{z}
\end{equation}
Indeed, this is a chief reason for interest in the $\mathcal H(b)$
spaces; the operators $X:=S^*|_{\mathcal H(b)}$ (and their analogs on
the vector-valued $\mathcal H(b)$ spaces) serve as functional models
for contractive operators on Hilbert space, see for example
\cite{BalKri,NikVas}.  We also note that, while $b$ itself may or may
not lie in $\mathcal H(b)$, it is always the case that $S^*b$ belongs
to $\mathcal H(b)$ \cite[II-8]{Sar-book}.

By the Herglotz formula, for each unimodular scalar $\alpha$ there is
a unique finite, positive measure $\mu_\alpha$ on the unit circle
$\mathbb T$ such that
\begin{equation}
 \frac{1+\alpha^*b(z)}{1-\alpha^*b(z)}=\int_{\mathbb T} \frac{1+z\zeta^*}{1-z\zeta^*}\, d\mu_\alpha(\zeta)
\end{equation}
The measures $\{\mu_\alpha \}_{\alpha\in\mathbb T}$ are called the
{\em Aleksandrov-Clark measure} or {\em AC measures} for $b$. We refer
to \cite[Chapter 9]{CimMatRos} for a discussion of their properties. Let
$P^2(\mu)$ denote the closure of the analytic polynomials in
$L^2(\mu)$.

By the ``elementary portions of Clark's theory'' we mean the following
three theorems adapted from \cite{Cla}.
\begin{thm}\label{thm:clark1}
For each $\alpha\in\mathbb T$, the formula
  \begin{equation}\label{eqn:cauchy-transform-def}
    (\mathcal V_\alpha f)(z) =(1-\alpha^*b(z))\int_{\mathbb T} \frac{f(\zeta)}{1-z\zeta^*}\,d\mu_\alpha(\zeta)
  \end{equation}
defines a unitary operator from $P^2(\mu)$ onto $\mathcal H(b)$. \cite[III-7]{Sar-book}
\end{thm}
To go further we assume that $b$ is an extreme point of the unit ball
of $H^\infty(\mathbb D)$; this is the case if and only if
$\int_{\mathbb T} \log(1-|b(\zeta)|)\, dm(\zeta)=-\infty$.  Since the
Radon-Nikodym derivative of the measure $\mu_\alpha$ is
\begin{equation}
\frac{d\mu_\alpha}{dm}(\zeta) = \frac{1-|b(\zeta)|^2}{|1-\alpha^*b(\zeta)|^2},  
\end{equation}
\cite[Proposition 9.1.14]{CimMatRos}, it follows that $\int_{\mathbb T} \log\left(\frac{d\mu_\alpha}{dm}\right)\,
dm=-\infty$. Thus by Szeg\H{o}'s theorem
$P^2(\mu_\alpha)=L^2(\mu_\alpha)$ for each $\alpha$, when $b$ is
extreme. In particular, in this case (and only this case) the isometry
$M_\zeta$ acting on $P^2(\mu_\alpha)$ is unitary. 
\begin{thm}\label{thm:clark2}
  Let $b$ be extreme and let $X$ denote the backward shift operator
  restricted to $\mathcal H(b)$.  Then for each $\alpha$ the rank-one
  perturbation
  \begin{equation}
    U_\alpha^*:=X+\alpha^* S^*b\otimes 1
  \end{equation}
defines a unitary operator $U_\alpha$. Each of these unitaries is
cyclic (with cyclic vector $1$), and the spectral measure of
$U_\alpha$ with respect to $1$ is the AC measure $\mu_\alpha$.
Moreover the operator $\mathcal V_\alpha$ implements the spectral
resolution of $U_\alpha$:
\begin{equation}
  U_\alpha \mathcal V_\alpha = \mathcal V_\alpha M_\zeta
\end{equation}
where $M_\zeta$ denotes multiplication by the independent variable in
$L^2(\mu_\alpha)$.\cite[III-8]{Sar-book}
\end{thm}
Finally, one can say something about eigenvalues of $U_\alpha$:
\begin{thm}\label{thm:clark3}

  The number $\zeta$ is an eigenvalue of $U_\alpha$ if and only if $b$
  has finite angular derivative at $\zeta$ with $b(\zeta)=\alpha$; in
  this case the eigenspace is one-dimensional and spanned by the
  function
  \begin{equation}
    k^b_\zeta(z):= \frac{1-\alpha^*b(z)}{1-z\zeta^*}
  \end{equation}
\cite[Theorem 8.9.9]{CimMatRos}
\end{thm}
Our goal, then, is to obtain analogs of
Theorems~\ref{thm:clark1}--\ref{thm:clark3} for deBranges-Rovnyak type
subspaces of $H^2_d$.  For this, we let $b$ be a {\em contractive
  multiplier} of $H^2_d$. That is, $b$ is an analytic function in the
ball such that $bf\in H^2_d$ whenever $f\in H^2_d$, and multiplication
by $b$ contracts norms:
\begin{equation}
  \label{eqn:b-contractive}
  \|bf\|_{H^2_d}\leq \|f\|_{H^2_d}.
\end{equation}
This operator is denoted $M_b$.  In one dimension, $b$ is a
contractive multiplier if and only if $b$ is bounded by $1$ in the
disk. In higher dimensions this condition is necessary but not
sufficient. Nonetheless the algebra of bounded multipliers of $H^2_d$
is in many ways a suitable analog of $H^\infty(\mathbb D)$ in higher
dimensions. It is true that $b$ is a contractive multiplier if and
only if the Hermitian kernel
\begin{equation}
  \label{eqn:hb-kernel-def}
  k^b(z,w):= \frac{1-b(z)b(w)^*}{1-{zw^*}}
\end{equation}
is positive semidefinite.  When this is the case it is the reproducing
kernel for a space $\mathcal H(b)$ which is contractively contained in
$H^2_d$.  Explicitly, $\mathcal H(b)$ is the range of the operator
$I-M_bM_b^*$ on $H^2_d$, equipped with the unique norm making this
operator a partial isometry. This norm is given by the expression
\begin{equation}\label{eqn:Hb-norm-def}
  \|f\|_{\mathcal H(b)}^2 = \sup_{g\in H^2_d} (\|f+bg\|_{H^2_d}^2 -\|g\|_{H^2_d}^2)
\end{equation}
(see \cite[Chapter I]{Sar-book}), though in this paper the description
of the space in terms of its kernel will be more useful.

The extension of the Clark theory to the $\mathcal H(b)$ spaces just
defined is not straightforward, for several reasons. First, the
obvious analog of the backward shift $S^*$ would be the $d$-tuple
adjoints of the coordinate multipliers $M_{z_1}, \dots M_{z_d}$ on $H^2_d$ (the
{\em $d$-shift} of Arveson \cite{Arv}, Drury \cite{Dru} and
M\"{u}ller-Vasilescu \cite{MulVas}). However the $\mathcal H(b)$ spaces
are in general not invariant for the adjoints of the $d$-shift
\cite{BBF1}.  Following \cite{BB,BBF1}, the correct operators to look
for are those that solve the {\em Gleason problem} in $\mathcal H(b)$.
That is, we seek operators $X_1, \dots X_d$ on $\mathcal H(b)$ such
that for all $f\in \mathcal H(b)$ we have
\begin{equation}
  \label{eqn:gleason-intro}
  f(z)-f(0) =\sum_{j=1}^d z_j (X_jf)(z)
\end{equation}
and such that the tuple $(X_1, \dots X_d)$ is contractive in the sense that 
\begin{equation}
  \sum_{j=1}^d \|X_jf\|^2 \leq 1-|f(0)|^2.
\end{equation}
for all $f\in \mathcal H(b)$.  (When $d=1$, the restricted backward
shift $X=S^*|_{\mathcal H(b)}$ always obeys this estimate, called the
``inequality for difference quotients'' in \cite{dBR-book}.) From
\cite{BB,BBF1} we know contractive solutions always exist, but a
principal difficulty is that, in general, such operators may not be
unique.

The next obstacle is understanding what (if anything) can play the
role of the AC measures $\mu_\alpha$.  First consider a finite,
positive measure $\mu$ on the unit sphere and define a function $b$ in
the ball by the formula
\begin{equation}
  \label{eqn:Mplus-def}
  \frac{1+b(z)}{1-b(z)}=\int_{\partial\mathbb B^d} \frac{1+z\zeta^*}{1-z\zeta^*}\, d\mu(\zeta)
\end{equation}
then $b$ will be a contractive multiplier of $H^2_d$, but importantly,
not every contractive multiplier admits such a representation
\cite{MccPut}.  The correct approach is to replace the
Herglotz-like kernel $\frac{1+z\zeta^*}{1-z\zeta^*}$ with the ``noncommutative'' Herglotz kernel 
\begin{equation}
  \label{eqn:nc-herglotz-intro}
  \left(I+\sum_{j=1}^d z_j L_j^*\right)\left(I-\sum_{j=1}^d z_j L_j^*\right)^{-1}
\end{equation}
where the $L_j$ are Hilbert space operators obeying the relations
\begin{equation}
  \label{eqn:cuntzrelations-intro}
  L_i^*L_j =\delta_{ij}I.
\end{equation}
The measure $\mu$ must then be replaced with a positive linear
functional on the operator system spanned by the NC Herglotz kernels
(\ref{eqn:nc-herglotz-intro}) and their adjoints. (Such NC Herglotz
kernels have been studied before, see
e.g. \cite{MccPut,Jur1,Pop-herglotz}.)

The remainder of the paper is organized as follows: in
Section~\ref{sec:herglotz} we reprove the NC Herglotz formula from
\cite{MccPut,Pop-herglotz} in the form in which we will need it,
define the noncommutative AC states $\{\mu_\alpha\}_{\alpha\in\mathbb
  T}$ associated to $b$, and use them to define (via a GNS type
construction) Hilbert spaces $P^2(\mu_\alpha)$. Using the NC Herglotz
kernel we are then able to construct a ``noncommutative normalized
Fantappi\`e transform'' $\mathcal V_\alpha$ which implements a unitary
equivalence between $P^2(\mu_\alpha)$ and $\mathcal H(b)$.  The
section concludes with Theorem~\ref{thm:Vmu}, which is our analog of
Theorem~\ref{thm:clark1}.

In Section~\ref{sec:moreGNS} we investigate the GNS construction in
the noncommutative $P^2(\mu)$ spaces more closely and introduce the
notion of a ``quasi-extreme'' multiplier $b$. It is these that will
substitute for the extreme points of the unit ball of
$H^\infty(\mathbb D)$.  We also introduce the coisometric $d$-tuples
of operators ${\bf S^\alpha}$ which are a partial analog of the
unitaries $U_\alpha$ in Clark's theory.  

In Section~\ref{sec:gleason} we consider the Gleason problem in
$\mathcal H(b)$ and prove the crucial result that, when $b$ is
quasi-extreme as defined in Section~\ref{sec:moreGNS}, there is in
fact a {\em unique} contractive solution ${\bf X}=(X_1, \dots X_d)$ to
the Gleason problem in $\mathcal H(b)$, and moreover for this solution
equality holds in the multivariable inequality for difference
quotients.  This result uses in a fundamental way the noncommutative
constructions of Section~\ref{sec:moreGNS}.  (In one variable there
would be nothing to do here, since the backward shift is trivially the
unique solution to the Gleason problem, regardless of $b$.) In
Section~\ref{sec:intertwining} we put the results of the previous two
sections together to show that there is a unique rank-one perturbation
of the (now unique) ``backward shift'' ${\bf X}$ that is unitarily
equivalent, via the NC Fantappi\`e transfrom, to the adjoint of the
GNS tuple, ${\bf S}^{\alpha *}$. This is Theorem~\ref{thm:clark-ball},
which is our extenstion of Theorem~\ref{thm:clark2}.

Finally, in Section~\ref{sec:spectra} we prove the analog of
Theorem~\ref{thm:clark3}, in which we show that for the GNS tuple
${\bf S^\alpha}$, the eigenvalue problem
\begin{equation}
  \label{eqn:eigenvalue-prob-intro}
  \sum_{j=1}^d \zeta_j^*S_j^\alpha h = h
\end{equation}
has a solution in $\mathcal H(b)$ if and only if $b$ has a finite
angular derivative at $\zeta\in\partial \mathbb B^d$ with
$b(\zeta)=\alpha$.  This is Theorem~\ref{thm:eigenvalues}. Along the
way we prove a number of other results, including a version of the
Aleksandrov disintegration theorem in this setting
(Theorem~\ref{thm:disintegrate}), and a characterization of those
$\mathcal H(b)$ spaces which are invariant under multiplication by the
coordinate functions $z_j$ (Corollary~\ref{cor:z-invariant}); it turns
out this is the case exactly when $b$ is not quasi-extreme.

\section{The NC Herglotz formula and NC Fantappi\`e transform}\label{sec:herglotz} 
\subsection{Row contractions, row isometries, and dilations}
We begin by recalling some basic constructions in multivariable
operator theory, in particular row isometries and the noncommutative
disk algebra of Popescu \cite{Pop-ncdisk}. Let $H$ be a Hilbert
space. A {\em row contraction} is a $d$-tuple of operators ${\bf
  T}=(T_1, \dots T_d)$ in $B(H)$ satisfying
\begin{equation}\label{eqn:rc-def}
\sum_{j=1}^n T_jT_j^* \leq I.
\end{equation}
In other words, the map
\begin{equation}
  (T_1, \dots T_d)\begin{pmatrix} h_1 \\ \vdots \\ h_d\end{pmatrix} =\sum_{j=1}^d T_jh_j
\end{equation}
is contractive from $H^d$ to $H$ (the direct sum of $d$ copies of
$H$), so when convenient we think of ${\bf T}$ as belonging to $B(H^d,
H)$.  Note that if ${\bf T}$ is isometric, then ${\bf T}^*{\bf
  T}=I_{H^d}$, and so
\begin{equation}\label{eqn:cuntzrels}
  T_i^*T_j =\delta_{ij}I_H,
\end{equation}
which says that the $T_j$ are isometries with orthogonal ranges.  If
${\bf T}$ is unitary, then also ${\bf T}{\bf T}^*=I_H$, which means
equality holds in (\ref{eqn:rc-def}).  In this case the $T_j$ are
called {\em Cuntz isometries}.  We will consider both commuting and
non-commuting row contractions; note however that if ${\bf T}$ is
isometric then the relations (\ref{eqn:cuntzrels}) show that the $T_j$
cannot commute.
\begin{defn}  Let ${\bf T}$ be a row contraction on $H$.  An {\em isometric
  dilation} of ${\bf T}$ is a row isometry ${\bf V}$ acting on a
  larger Hilbert space $K\supset H$ such that for each $j=1,\dots d$,
  the space $H$ is invariant for $V_j^*$ and $V_j^*|_H =T_j^*$.  The
  dilation is called {\em unitary} (or a {\em Cuntz dilation}) if the
  $V_j$ are Cuntz isometries.  The dilation is called {\em minimal} if
\begin{equation}
K=\bigvee_{w\in \mathbb F_d^+} V_w H.
\end{equation}
\end{defn}
(That is, the smallest ${\bf V}$-invariant subspace containing $H$ is
$K$ itself.)  By results of Frazho \cite{Fra}, Bunce \cite{Bun}, and
Popescu \cite{Pop-dil} every row contraction admits a minimal
isometric dilation, which is unique up to unitary equivalence.  More
precisely, if ${\bf T}$ is a row contraction and ${\bf V}$ and ${\bf V}^\prime$ are
minimal isometric dilations of ${\bf T}$ on Hilbert spaces $K\supset H$,
$K^\prime \supset H$ respectively, then the map
\begin{equation}
  UV_wh :=V^\prime_wh
\end{equation}
extends to a unitary transformation from $K$ onto $K^\prime$ satisfying $UV_j=V^\prime_jU$ for all $j$.

\subsection{The noncommutative disk algebra}\label{sec:nc-disk} 
Let $\Fd$ denote the {\em free semigroup on $d$ letters}, that is, the
set of all finite words
\begin{equation}
  w=i_1i_2\cdots i_m
\end{equation}
where $m\geq 1$ is an integer, and the $i_j$ are drawn from the set
$\{1, 2, \dots d\}$.  We also include in $\Fd$ the {\em empty word}
$\varnothing$.  The integer $m$ is called the {\em length} of the word
$w$, by convention the empty word has length zero. If $w, v\in \Fd$
are words of lengths $m$ and $n$ respectively, they can be
concatenated to produce the word $wv$, of length $m+n$. By convention
$\varnothing w =w\varnothing=w$ for all $w\in \Fd$.  The {\em full
  Fock space} $\mathcal F_d$ is the Hilbert space $\ell^2(\Fd)$, with
orthonormal basis $\{\xi_w\}_{w\in\Fd}$. The Hilbert space $\mathcal
F_d$ supports bounded operators $L_1, \dots L_d$, defined by their
action on the basis vectors $\xi_w$:
\begin{equation}
  L_i \xi_w = \xi_{iw}.
\end{equation}
It is straightforward to check that the $L_i$ satisfy
\begin{equation}\label{eqn:Li-ortho}
  L_j^*L_i =\delta_{ij}I.
\end{equation}
Thus by the discussion above $\L$ is a row isometry. In particular,
this implies that ${\bf L}{\bf L}^* = \sum_{i=1}^d L_iL_i^*$ is an
orthogonal projection in $\mathcal F_d$; the range of this projection
is the orthogonal complement of the one-dimensional space spanned by
the vacuum vector $\xi_{\varnothing}$.

 The {\em noncommutative disk algebra} $\mathcal A_d$ (we will fix $d$
 and abbreviate this to $\mathcal A$ is the
 norm-closed algebra of operators on $\mathcal F_d$ generated by $L_1,
 \dots L_d$ and the identity operator $I$.  We will write $\mathcal
 A^*$ for the algebra of operators which are the adjoints of the
 operators in $\mathcal A$.  The C* -algebra generated by the $L_i$
 is called the {\em Cuntz-Toeplitz algebra} $\mathcal E_d$.  The norm
 closure of $\mathcal A_d +\mathcal A_d^*$ in $\mathcal E_d$ is called
 the {\em Cuntz-Toeplitz operator system}. (Recall that an operator
 system is a unital, self-adjoint linear subspace of a unital
 C*-algebra.)  A theorem of Popescu \cite{Pop-opsystem} shows that the row
 isometry $\L$ is universal, in the following sense: if $(V_1, \dots
 V_d)$ is any row isometry, acting on a Hilbert space $H$, then there
 is a representation of the Cuntz-Toeplitz algebra $\pi:\mathcal E_d\to
 B(H)$ such that $\pi(L_i)=V_i$ for each $i=1, \dots d$. More
 generally, if ${\bf T}=(T_1, \dots T_d)$ is any row contraction
 acting in a Hilbert space $H$, then there is a unital, completely
 positive map $\rho:\mathcal E_d\to B(H)$ such that $\rho(L_j)=T_j$. 

A particular subsystem of the Cuntz-Toeplitz operator system will be
of interest.  For a $d$-tuple of nonnegative integers $\nn =(n_1,\dots
n_d)$, and an arbitrary $d$-tuple of operators ${\bf T}=(T_1, \dots
T_d)$ we define the {\em symmetrized monomials}
\[
T^{(\nn)} := \sum T_{i_1}\cdots T_{i_{|\nn|}}
\]
where the sum is taken over all products of exactly $n_1$ $T_1$'s,
$n_2$ $T_2$'s, etc.  So, for example if $d=2$ and $\nn=(2,1)$, then
\[
T^{(2,1)} = T_1^2 T_2  +T_2 T_1^2 + T_1T_2T_1
\]
By convention we put $T^{({\mathbf 0})}=I$.  In particular, if
$z=(z_1,\dots z_d)$ is a $d$-tuple of scalars, and the monomial
$z^{\nn}$ is defined in the usual multi-index notation as
$z^{\nn}=z_1^{n_1} \cdots z_d^{n_d}$, then we have for each integer
$k\geq 1$
\begin{equation}\label{E:zL_power}
(z_1T_1 +\cdots +z_dT_d)^k = \sum_{|\nn|=k} z^{\nn} T^{(\nn)}.
\end{equation}

The {\em symmetric part} $\mathcal S$ of $\mathcal A$ is defined to be
the closed linear span of the symmetrized monomials $\{L^{(\nn)} :
\nn\in \mathbb N^d\}$.  Much of our interest will be in positive
linear functionals $\mu$ defined on the operator system $\mathcal
S+\mathcal S^*\subset \mathcal A+\mathcal A^*$.  

In what follows we will use the notation
\begin{equation}
  \zlstar:= \sum_{j=1}^d z_j L_j^*.
\end{equation}
It follows that for all $z,w\in\mathbb C^d$,
\begin{equation}\label{eqn:KEY}
  (\zlstar)( {{\bf L}w^*}) = \sum_{j,k=1}^d z_jw_l^* L_j^*L_k ={zw^*}
\end{equation}
by the orthogonality relations for the $L_j$. In particular by putting
$z=w$ we have
\begin{equation}
  \|\zlstar\|=|z|
\end{equation}
and hence for all $z\in\mathbb B^d$ the operator $I-\zlstar$ is
invertible, with inverse given by the (norm-convergent) series
\begin{equation}
  (I-\zlstar)^{-1} =\sum_{k=0}^\infty (\zlstar)^k =
  \sum_{\nn\in\mathbb N^d} z^{\nn}\frac{\nn!}{|\nn|!}L^{(\nn)*}.
\end{equation}
It follows that $(I-\zlstar)^{-1}$ belongs to $\mathcal S^*$ for all
$z\in\mathbb B^d$.  

The identity (\ref{eqn:KEY}) explains the appearance of noncommutative
methods in our treatment of the $\mathcal H(b)$ spaces in $H^2_d$.
Note that in one variable, if $z,w$ are complex numbers and
$|\zeta|=1$ then trivially
\begin{equation}
(z\zeta^*)(\zeta w^*)=zw^*  
\end{equation}
However if $z,w\in \mathbb C^d$, $d>1$, and $\zeta\in\mathbb C^d$ has
unit norm, then
\begin{equation}\label{eqn:badKEY}
(z\zeta^*)(\zeta w^*)  = \sum_{j,k=1}^d z_jw_k^*\zeta_j^*\zeta_k\neq zw^*.
\end{equation}
By replacing $\zeta$ with the row isometry ${\bf L}$, equation
(\ref{eqn:KEY}) ``repairs'' equation (\ref{eqn:badKEY}).  (Indeed,
note that the identity $z{\bf T^*}{\bf T}w^*=zw^*$ cannot hold for any
commuting tuple ${\bf T}$ when $d>1$.)  The identity (\ref{eqn:KEY})
is thus central to our development, especially in the proof of the key
algebraic results in Proposition~\ref{prop:NC-fantappie}.

The following lemma will be used several times.
\begin{lem}\label{lem:NC-kernels-dense}
  The linear span of the set
  \begin{equation}
    \{ (I-{\bf L}w^*)^{-1}:w\in\mathbb B^d\}
  \end{equation}
is norm dense in $\mathcal S$.
\end{lem}
\begin{proof}
  Let $\mathcal M$ denote the closed linear span of the $(I-{\bf
    L}w^*)^{-1}$ in $\mathcal S$ as $w$ ranges over $\mathbb B^d$.
  First, note that if $T$ is any operator on Hilbert space with
  $\|T\|\leq 1$, then by expanding $(I-rT)^{-1}$ in a geometric
  series, we have for each positive integer $m$
  \begin{equation}
    T^m =\lim_{r\to 0} \frac{1}{r^m} \left( (I-rT)^{-1} -\sum_{n=0}^{m-1}r^nT^n \right)
  \end{equation}
  where the limit exists in the operator norm.  Since $I\in \mathcal
  M$, induction on this fact with $T={\bf L}w^*$ shows that $({\bf
    L}w^*)^{m}$ lies in $\mathcal M$ for all $w\in\mathbb B^d$ and all
  $m\geq 0$.

  From this, it suffices to prove that for each fixed $m$, the span of
  $\{({\bf L}w^*)^m:|w|<1\}$ is equal to the span of the set
  $\{L^{(\pp)} :|\pp| =m\}$.  From (\ref{E:zL_power}), the former span
  is contained in the latter. If they are not equal, then by linear
  algebra there is a set of scalars $\{c_{\pp}:|\pp| =m\}$, not all
  $0$, so that
  \begin{equation}
    \sum_{|\pp|=m} c_{\pp} w^{\pp} =0
  \end{equation}
  for all $|w|<1$. But this is a polynomial which vanishes on the open
  ball $\mathbb B^d$, and hence must vanish identically, so all
  $c_{\pp}$ are $0$, a contradiction.
\end{proof}

The next lemma encodes a key observation used in what will follow,
namely that if $p, q$ are polynomials in $\mathcal S$, then $p({\bf
  L})^*q({\bf L})$ belongs to $\mathcal S+\mathcal S^*$.  (This is
essentially an elaboration of (\ref{eqn:KEY}) which will
allow us to carry out a GNS-type construction in $\mathcal S+\mathcal
S^*$ in Section~\ref{sec:GNS-main}.) We will do the required calculation quite
explicitly.  First, some notation: for $d$-tuples of nonnegative
integers $\mm=(m_1, \dots m_d)$, $\nn=(n_1, \dots n_d)$, say $\mm\leq
\nn$ if and only if $m_i\leq n_i$ for each $i=1, \dots d$.  If
$\mm\leq \nn$, define $\nn-\mm=(n_1-m_1, \dots n_d-m_d)$.

Next, we introduce the {\em letter counting map} $\lambda:\Fd\to
\mathbb N^d$, which when applied to a word $w$ returns the $d$-tuple
$(n_1, \dots n_d)$ where $n_1$ is the number of $1$'s appearing in
$w$, $n_2$ the number of $2$'s, etc.  It is immediate from definitions
that
\begin{equation}
  L^{(\nn)} =\sum_{\lambda(w)=\nn} L_w.
\end{equation}
\begin{lem}\label{lem:bimodule}
  For all $\mm,\nn\in \mathbb N^n$, 
  \begin{equation}
    L^{(\nn)*}L^{(\mm)} = \begin{cases}
\frac{|\nn|!}{\nn !}L^{(\mm-\nn)} &\text{if } \mm\geq \nn \\      
\frac{|\mm|!}{\mm !}L^{(\nn-\mm)*} & \text{if } \nn\geq \mm \\
\frac{|\nn|!}{\nn !}I &\text{if } \mm=\nn \\
      0 &\text{otherwise,}
    \end{cases}
  \end{equation}
and hence if $p,q$ are polynomials in $\mathcal S$ then $p({\bf
  L})^*q({\bf L})$ lies in $\mathcal S+\mathcal S^*$.
\end{lem}
\begin{proof}
First suppose $\mm\geq \nn$.  Fix $w$ with $\lambda(w)=\nn$.  Let
$E(w)$ denote the set of words in $\lambda^{-1}(\mm)$ whose initial
string is $w$:
  \begin{equation}
    E(w) =\{ u\in \Fd | u=wv \text{ and } \lambda(u)=\mm\}.
  \end{equation}
Note that this set is alternatively defined as
\begin{equation}
E(w)=\{wv\in\Fd | \lambda(v)=\mm-\nn\}  
\end{equation}
Now, if $u\in\lambda^{-1}(\mm)$, then $L_w^*L_u=L_v$ if $u\in E(w)$
and $u=wv$, while $L_w^*L_u=0$ if $u\notin E(w)$. Thus
\begin{align}
  L^{(\nn)*}L^{(\mm)} &= \sum_{\lambda(w)=\nn} \sum_{\lambda(u)=\mm} L_w^* L_u \\
&= \sum_{\lambda(w)=\nn} \sum_{u=wv\in E(w)} L_w^*L_u \\
&= \sum_{\lambda(w)=\nn} \sum_{v\in\lambda^{-1}(\nn-\mm)} L_v \\
&= \frac{|\nn|!}{\nn !}L^{(\mm-\nn)}
\end{align}
since the cardinality of $\lambda^{-1}(\nn)$ is $\frac{|\nn|!}{\nn
  !}$.  The cases $\nn\geq \mm$ and $\nn=\mm$ follow by symmetry.

Finally, if $\mm$ and $\nn$ are incomparable, then no word in
$\lambda^{-1}(\mm)$ is a subword of a word in $\lambda^{-1}(\nn)$ and
vice versa, so each summand $L_w^*L_u$ is $0$.
\end{proof}
\subsection{The space $P^2(\mu)$}\label{sec:p2mu} Now, if $\mu$ is a positive linear functional on $\mathcal
S^*+\mathcal S$, Lemma~\ref{lem:bimodule} allows us to define a
pre-inner product on $\mathcal S\times\mathcal S$: for polynomials
$p,q\in \mathcal S$, define
\begin{equation}
  \langle p,q\rangle := \mu(q({\bf L})^*p({\bf L})).
\end{equation}
Since $\mu$ is a positive linear functional, this map obeys the
Cauchy-Schwarz inequality
\begin{equation}\label{eqn:CS-for-mu}
|\mu(q({\bf L})^*p({\bf L}))|^2\leq \mu(q({\bf L})^*q({\bf L}))\mu(p({\bf L})^*p({\bf L}))  
\end{equation}
and thus extends to a pre-inner product on all of $\mathcal
S\times\mathcal S$.  We will write $P^2(\mu)$ for the Hilbert space
obtained by modding out null vectors and completing, and for $p\in
\mathcal S$ we write $[p]$ for the image of $p$ in $P^2(\mu)$.

\subsection{de Branges-Rovnyak spaces and the ``noncommutative'' AC measures}
Let $b$ be a contractive multiplier of $H^2_d$ (hereafter we will just
call $b$ a {\em multiplier}). From the introduction, we have the
reproducing kernel Hilbert space $\mathcal H(b)$ with kernel
\begin{equation}
  k^b(z,w)=\frac{1-b(z)b(w)^*}{1-{zw^*}}
\end{equation}
Central to our development will be the following noncommutative
Herglotz-style formula for $b$. Such a formula is established in
\cite{MccPut} and \cite{Pop-herglotz}, we include a proof here since
it is short (and to establish the role of the operator system
$\mathcal S+\mathcal S^*$).  The formula is based on the NC Herglotz
kernel
\begin{equation}
  H(z,L):= (I+\zlstar)(I-\zlstar)^{-1}
\end{equation}
For each $z$ in the ball, the operator $H(z,L)$ has positive real
part, indeed using the relation (\ref{eqn:KEY}) one finds
\begin{equation}
  H(z,L)+H(z,L)^* = (1-|z|^2) (I-\zlstar )^{-1}(I-\zlstar)^{*-1}.
\end{equation}
\begin{prop}\label{prop:nc-herglotz}
Let $b$ be a contractive multiplier of the Drury-Arveson space
$H^2_d$. Then there exists a unique positive linear functional $\mu$
on $\mathcal S+\mathcal S^*$ such that
\begin{equation}\label{eqn:nc-herglotz}
  \frac{1+b(z)}{1-b(z)} = \mu\left((I+\zlstar)(I-\zlstar)^{-1} \right) +i\text{Im} \left(\frac{1+b(0)}{1-b(0)} \right).
\end{equation}
\end{prop}
\begin{proof}
Consider the analytic function
\[
f(z)=(1+b(z))(1-b(z))^{-1}
\]
and observe that $f$ belongs to the {\em positive Schur class}, i.e. the kernel 
\[
\frac{f(z)+f(w)^*}{1-{zw^*}}
\]
is positive.  Indeed it factors as
\[
(1-b(z))^{-1}\frac{1-b(z)b(w)^*}{1-{zw^*}}(1-b(w)^*)^{-1}
\]
We may thus factor
\begin{equation*}
f(z)+f(w)^*  = F(z)[1-{zw^*}]F(w)^*
\end{equation*}
for a holomorphic function $F$ taking values in some auxiliary
Hilbert space $H$.  Substituting in turn $w=0$, $z=0$, and $z=w=0$, we get
\begin{align*}
f(z)+f(0)^*  &= F(z)F(0)^*\\
f(0)+f(w)^*  &= F(0)F(w)^*\\
f(0)+f(0)^*  &= F(0)F(0)^*
\end{align*}
Adding the first and last equations and subtracting the middle two leaves
\[
F(z){zw^*} F(w)^* =[F(z)-F(0)][F(w)-F(0)]^*
\]
By the lurking isometry argument, there exists an isometric $d$-tuple
${\bf V}=(V_1,\dots V_d)$ on $H$ such that
\[
\sum_{j=1}^d w_j^*V_j F(w)^* =F(w)^*-F(0)^*
\] 
Solving for $F(w)^*$ gives
\[
F(w)^* = (I-w^*{\bf V})^{-1} F(0)^*
\]
or
\begin{equation}\label{eqn:nc-herglotz-step}
f(z)+f(0)^* = F(z)F(0)^* = F(0)(I-z{\bf V}^*)^{-1} F(0)^*
\end{equation}
By the Frazho-Bunce-Popescu dilation theorem, the tuple ${\bf V}$ is the
image of ${\bf L}$ under a unital, completely positive map $\psi$.  We now define
\begin{equation}
  \mu(L^{(\nn)})=F(0)\psi(L^{(\nn)})F(0)^*
\end{equation}
which shows that $\mu$ is positive, since $\psi$ is positive. With
this definition and some algebra, (\ref{eqn:nc-herglotz-step}) becomes
\begin{equation}\label{eqn:nc-herglotz-final}
f(z) = \mu\left((I+\zlstar)(I-\zlstar)^{-1} \right) +i\text{Im}f(0)  
\end{equation}
as desired.  The uniqueness of $\mu$ is clear, since by
(\ref{eqn:nc-herglotz-final}) the value of $\mu(L^{(\nn)})$ is just
the coefficient of $z^{(\nn)}$ in the Taylor expansion of $f$ (with
$\mu(I)=\text{Re}f(0)$ when $\nn=0$).
\end{proof}
This proposition has a converse; namely if $\mu$ is a positive
functional on $\mathcal S+\mathcal S^*$ and $b$ is defined by
(\ref{eqn:nc-herglotz}) then $b$ is a contractive multiplier of
$H^2_d$; it follows as in \cite{MccPut} by reversing the steps of the
above argument. The principal reason for introducting $\mathcal
S+\mathcal S^*$ is that it forces the functional $\mu$ representing
$b$ to be unique; this need not be the case if we worked with
$\mathcal A+\mathcal A^*$ (or, say, the whole Cuntz-Toeplitz algebra).

With $b$ fixed and $\alpha$ a unimodular scalar, we can carry out the
above construction with $\alpha^*b$ in place of $b$. We then have
\begin{defn}\label{defn:ac-state-def} Let $b$ be a contractive
  multiplier of $H^2_d$. The {\em Aleksandrov-Clark state} (or {\em AC
    states}) for $b$ are the family of states
  $\{\mu_\alpha\}_{\alpha\in\mathbb T}$ on $\mathcal S+\mathcal S^*$
  such that
  \begin{equation}\label{eqn:ac-state-def}
  \frac{1+\alpha^*b(z)}{1-\alpha^*b(z)} = \mu_\alpha\left((I+\zlstar)(I-\zlstar)^{-1} \right) +i\text{Im} \left(\frac{1+\alpha^*b(0)}{1-\alpha^*b(0)} \right)
\end{equation}
as in Proposition~\ref{prop:nc-herglotz}.
\end{defn}

If we compare the Herglotz-type formula (\ref{eqn:nc-herglotz}) with
the classical one-variable formula
\begin{equation}\label{eqn:1d-herglotz}
  \frac{1+b(z)}{1-b(z)} =\int_{\mathbb T} \frac{1+z\zeta^*}{1-z\zeta^*}\, d\mu(\zeta) +i\text{Im} \frac{1+b(0)}{1-b(0)}
\end{equation}
this suggests viewing the expression 
\begin{equation}
  (I+\zlstar)(I-\zlstar)^{-1}
\end{equation}
as a noncommutative Herglotz kernel, and
\begin{equation}
  (I-\zlstar)^{-1}
\end{equation}
as a noncommutative Szeg\H{o} kernel. This is explored further in the next section. 

\subsection{The NC Fantappie transform}

Consider again the one variable case.  As noted in the introduction,
if $\mu$ is an AC measure for $b$, then Theorem~\ref{thm:clark1} says
that the normalized Cauchy transform
\begin{equation}
\mathcal V_\mu(f)(z):= (1-b(z))\int_{\mathbb T}
\frac{f(\zeta)}{1-z\zeta^*}\, d\mu
\end{equation}
implements a unitary operator from $P^2(\mu)$ onto $\mathcal
\mathcal{H}(b)$.  We are now ready to prove the analog of this theorem
in the ball.
\begin{defn}
  Let $\mu$ be a state on $\mathcal S+\mathcal S^*$, representing a
  multiplier $b$.  For a polynomial $p\in\mathcal S$, the {\em
    normalized NC Fantappi\`e transform} of $p$ is
  \begin{equation}\label{eqn:NC-transform-def}
    \mathcal V_\mu(p)(z) := (1-b(z))\mu((1-\zlstar)^{-1} p({\bf L})).
  \end{equation}
\end{defn}
Using Lemma~\ref{lem:bimodule} and the fact that the series expansion
of $(1-\zlstar)^{-1}$ is norm convergent in $\mathcal S^*$, one sees
that $(1-\zlstar)^{-1} p({\bf L})$ belongs to the closure of $\mathcal
S+\mathcal S^*$, so $\mathcal V_\mu$ is defined. Our next goal is to
show that $\mathcal V_\mu$ extends to a unitary operator from
$P^2(\mu)$ onto $\mathcal H(b)$.

We will also use the notation
\begin{equation}
  \label{eqn:nc-cauchy-transform-def}
  \mathcal K_\mu p(z):= \mu((1-\zlstar)^{-1} p({\bf L}))
\end{equation}
so that $\mathcal V_\mu p =(1-b(z))\mathcal K_\mu p$. Once we show
that $\mathcal V_\mu$ extends to a unitary operator on $P^2(\mu)$, it
follows that $\mathcal K_\mu$ also extends to a well-defined linear
operator taking $P^2(\mu)$ into the space of holomorphic functions on
the ball.

To streamline the notation, write
\begin{equation}
  H(z,L) =\Fzplus \Fz.
\end{equation}
\begin{prop}\label{prop:NC-fantappie}
  For all $z,w\in\mathbb B^d$, 
  \begin{equation}\label{eqn:NC-kernel-identity}
    \Fz \Fw = \frac12\left(\frac{H(z,L) + H(w,L)^* }{1-{zw^*} }\right)
  \end{equation}
In particular, if $\mu$ is a positive linear functional on $\mathcal
S+\mathcal S^*$ and $\mu$ represents $b$ as in
(\ref{eqn:nc-herglotz}), then
\begin{align}
\label{eqn:NC-db-kernel} \mu(\Fz\Fw) &= \frac12 \frac{1}{1-{zw^*}} \left( \frac{1+b(z)}{1-b(z)}+\frac{1+b(w)^*}{1-b(w)^*}\right)\\
\label{eqn:NC-db-kernel-2}&=\frac{1}{(1-b(z))(1-b(w)^*)}\left( \frac{1-b(z)b(w)^*}{1-{zw^*}}\right)
\end{align}
\end{prop}
\begin{proof} Working with the right-hand side of (\ref{eqn:NC-kernel-identity}), factor out $\Fz$ from the left and $\Fw$ from the right, leaving
\begin{align}
   \begin{split}\frac12 &\frac{H(z,L) + H(w,L)^* }{1-{zw^*} } \\ &= \Fz \left(\frac12 \frac{\Fzplus\Fwminus + \Fzminus \Fwplus }{1-{zw^*}} \right)  \Fw \end{split}\\
&= \Fz \frac12 \frac{2(I-(\zlstar)({\bf L}w^*))}{1-{zw^*}} \Fw \\ &=
   \Fz \Fw
\end{align}
where the last equality follows from (\ref{eqn:KEY}).  Equations (\ref{eqn:NC-db-kernel}) and (\ref{eqn:NC-db-kernel-2}) follow immediately.
\end{proof}

\begin{thm}\label{thm:Vmu}
  Let $b$ be a multiplier with AC state $\mu$.  Then the normalized NC
  Fantappi\`e transform $\mathcal V_\mu$ extends to a unitary operator
  from $P^2_\mu$ onto $\mathcal{H}(b)$.
\end{thm}
\begin{proof}
For each $w\in\mathbb B^n$, define 
\begin{equation}
  G_w(L)=(1-b(w)^*) (I-{{\bf L}w^*})^{-1}.  
\end{equation}
Let us write $[G_w]$ for the vector in $P^2(\mu)$ associated to $G_w$
in the construction of $P^2(\mu)$. By
Lemma~\ref{lem:NC-kernels-dense}, the span of the $[G_w]$ is dense in
$P^2(\mu)$. Then (\ref{eqn:NC-db-kernel-2}) shows that $\langle
[G_w],[G_z]\rangle_\mu = \langle k^b_w,k^b_z\rangle_{\mathcal
  \mathcal{H}(b)}$ for all $z,w\in\mathbb B^n$, so the map sending
$G_w$ to $k^b_w$ is an isometry from the span of the $G_w$ onto the
span of the $k^b_w$, and thus extends uniquely to a unitary from
$P^2_\mu$ onto $\mathcal{H}(b)$.  But by (\ref{eqn:NC-db-kernel-2})
again, the map sending $G_w$ to $k^b_w$ just is the normalized NC
Fantappi\`e transform.
\end{proof}

On $\mathcal A+\mathcal A^*$ there is a distinguished state called the
{\em vacuum state}, which is the vector state induced by the vacuum
vector $\xi_\varnothing$. That is, for polynomials $p,q\in\mathcal A$
we define
\begin{equation}\label{eqn:vacuum-state-def}
  m_{\varnothing}(p+q^*) := \langle (p({\bf L})+q({\bf L})^*)\xi_{\varnothing},\xi_\varnothing\rangle.
\end{equation}
Inspecting the moments we find that, since $\xi_{\varnothing}$ is a
wandering vector for ${\bf L}$, we have $m_{\varnothing}(I)=1$ and
$m_\varnothing (L_w)=0$ for $w\neq \varnothing$.  Thus $m_\varnothing$
can be thought of as an analogue of Lebesgue measure $m$, which is the
measure on $\mathbb T$ (or, state on $C(\mathbb T)$) with moments
$\widehat{m}(1)=1$ and $\widehat{m}(z^n)=0$ for $n\neq 0$. The analogy
is strengthened by noting that if we restrict $m_\varnothing$ to
$\mathcal S+\mathcal S^*$, then $m_\varnothing$ is an AC state for
$b\equiv 0$, and hence $\mathcal H(b)$ is exactly the Drury-Arveson
space $H^2_d$. Explicitly, Theorem~\ref{thm:Vmu} applied to the
function $b\equiv 0$ with associated state $m_\varnothing$ says
\begin{equation}
  \frac{1}{1-{zw^*}} =\langle k_w, k_z\rangle_{H^2_d} = m_\varnothing((I-\zlstar)^{-1}(I-{{\bf L}w^*})^{-1})
\end{equation}
which can be compared to the classical one variable identity
\begin{equation}\label{eqn:szego-inner-prod}
  \frac{1}{1-zw^*}=\int_{\mathbb T}
  \frac{1}{1-z\zeta^*}\frac{1}{1-\zeta w^*}\, dm(\zeta).
\end{equation}
More generally, the equation (\ref{eqn:NC-db-kernel-2}) is in one variable the identity
\begin{equation}
  \frac{1}{(1-b(z))(1-b(w)^*)}\frac{1-b(z)b(w)^*}{1-zw^*} =
  \int_{\mathbb T} \frac{1}{1-z\zeta^*}\frac{1}{1-\zeta w^*}\, d\mu(\zeta)
\end{equation}
(see \cite[III-6]{Sar-book}). Indeed the identity (\ref{eqn:KEY})
means that the proofs given in this section reduce to those of \cite[Chapter
III]{Sar-book} when $d=1$.

Even more, the vacuum state $m_\varnothing$ supports a version of the
Aleksandrov disintegration theorem for the AC states $\mu_\alpha$
associated to a fixed $b$ (Definition~\ref{defn:ac-state-def}). Indeed
the proof in our setting is essentially the same as that given in
\cite[Theorem 9.3.2]{CimMatRos} in the one-variable case.

\begin{thm}[Aleksandrov disintegration for AC
  states]\label{thm:disintegrate} Let $m$ denote normalized Lebesgue
  measure on $\mathbb T$, $m_\varnothing$ the vacuum state on
  $\overline{\mathcal S+\mathcal S^*}$, and
  $\{\mu_{\alpha}\}_{\alpha\in\mathbb T}$ the AC states for a
  contractive multiplier $b$.  Then for all $f\in \overline{\mathcal
    S+\mathcal S^*}$, the function $\alpha\to \mu_\alpha(f)$ is
  continuous in $\alpha$, and
  \begin{equation}\label{eqn:disintegrate}
    \int_{\mathbb T} \mu_\alpha(f)\, dm(\alpha) = m_{\varnothing}(f).
  \end{equation}
\end{thm}
\begin{proof}
  Using the positivity of the $\mu_\alpha$ and
  Lemma~\ref{lem:NC-kernels-dense}, it suffices to prove the theorem
  when $f=(I-\zlstar)^{-1}$ for fixed $|z|<1$.  In this case by
  (\ref{eqn:NC-db-kernel-2}) we have
  \begin{equation}\label{eqn:1balpha}
    \mu_\alpha(f) =\frac{1-b(z)b(0)^*}{(1-\alpha^*b(z))(1-\alpha b(0)^*)}
  \end{equation}
  which is continuous in $\alpha$ (note $z$ is fixed here and
  $|b(z)|<1$). On the one hand, by definition of $m_\varnothing$ we
  have $m_\varnothing(f)=1$.  On the other hand, integrating
  (\ref{eqn:1balpha}) we have (using the classical formula
  (\ref{eqn:szego-inner-prod}) for the inner product of Szeg\H{o}
  kernels)
\begin{align}
  \int_{\mathbb T} \frac{1-b(z)b(0)^*}{(1-\alpha^*b(z))(1-\alpha b(0)^*)}\, dm(\alpha) &= (1-b(z)b(0)^*) \int_{\mathbb T} \frac{1}{1-\alpha^*b(z)}\frac{1}{1-\alpha b(0)^*}\, dm(\alpha) \\
&= \frac{1-b(z)b(0)^*}{1-b(z)b(0)^*} = 1.
\end{align}
We conclude that
\begin{equation}
  \int_{\mathbb T} \frac{1}{1-b(z)\alpha^*} \, dm(\alpha) = 1 = m_\varnothing(f).
\end{equation}
\end{proof}

\section{The GNS construction in $P^2(\mu)$}\label{sec:moreGNS}  In
this section we carry out a version of the GNS construction in the
noncommutative $P^2(\mu)$ spaces of Section~\ref{sec:p2mu}. This
construction and the notions arising out of it (particularly that of a
{\em quasi-extreme} multiplier) will be central to the rest of the
paper.  In one variable, if $\mu$ is a measure on the circle then
multiplication by the independent variable $\zeta$ is an isometric
operator on $P^2(\mu)$, which is unitary in the case that
$P^2(\mu)=L^2(\mu)$ (equivalently, $P^2_0(\mu)=P^2(\mu)$).  In the
present setting the fact that $\mathcal S$ (the symmetric part of the
NC disk algebra) is not an algebra will complicate matters. In the end
we will obtain a contractive tuple ${\bf S}$ acting on a closed
subspace $P^2_0(\mu)$ of $P^2(\mu)$, which will be coisometric in the
case that $P^2(\mu)=P^2_0(\mu)$.

The GNS construction for states on the full Cuntz-Toeplitz operator
system $\mathcal A+\mathcal A^*$ is well known; we recount it briefly.
Suppose $H$ is a Hilbert space and $A\subset B(H)$ is a linear
subspace containing $I$.  Then
\begin{equation}\label{eqn:astar-plus-a}
  A^*+A =\{ b^*+a |a,b\in A\} 
\end{equation}
and
\begin{equation}\label{eqn:star-a}
A^*A={\rm span } \{b^*a | a,b\in A\}  
\end{equation}
are operator systems containing $A$, and since $A$ is unital we have $A^*+A\subseteq A^*A$.  

For $A=\mathcal A$, the noncommutative disk algebra, consider the
operator system $\mathcal M=\mathcal A^*+\mathcal A$.  One sees easily
from the relations (\ref{eqn:cuntzrels}) that for all words $w,v$, the
operator $L_w^*L_v$ belongs either to $\mathcal A$ or $\mathcal A^*$.
It follows that
\begin{equation}
  \mathcal A^*\mathcal A\subset {\overline{\mathcal A^*+\mathcal A}}
\end{equation}
This fact allows us to construct a ``left regular representation'' of
$\mathcal A$ starting from any state $\nu$ on $\mathcal A+\mathcal
A^*$.  (Here we abuse the terminology slightly and allow ``state'' to
mean any positive linear functional; it need not be normalized to have
$\nu(I)=1$.) In what follows we will often elide the distinction
between positive functionals on $\mathcal A+\mathcal A^*$ and their
unique extensions to positive functionals on $\overline{\mathcal
  A+\mathcal A^*}$; in practice this should cause no difficulties. A
similar remark will of course be in force for $\mathcal S+\mathcal S^*$.

Given $\nu$, the pairing on $\mathcal A\times \mathcal A$ given by
\begin{equation}\label{eqn:gns-pairing-M}
  \langle b,c\rangle:= \nu(c^*b)
\end{equation}
is a pre-inner product on $\mathcal A$; quotienting by null vectors
and completing gives a Hilbert space $H_\nu$.  For $a\in \mathcal A$,
let $[a]$ denote the corresponding vector in $H_\nu$.  Now it is
routine to check that for each $a\in\mathcal A$, the equation
\begin{equation}\label{eqn:lrrep-def}
  \pi(a)[b]:=[ab]
\end{equation}
(that is, ``left multiplication by $a$'') defines a bounded linear
operator on $H_\nu$, and the map $\pi: \mathcal A\to B(H_\nu)$ is a
completely contractive unital homomorphism.  Moreover, it is not hard
to show that the $d$-tuple $\pi(\L)=(\pi(L_1),\dots \pi(L_d))$ is a
row isometry.  Indeed, we have for all $b,c\in \mathcal A$ and all $i,
j=1,\dots d$
\begin{align}
\langle \pi(L_i)^* \pi(L_j)[b],[c]\rangle &= \langle \pi(L_j)[b],\pi(L_i)[c]\rangle \\
&= \nu(c^*L_i^*L_jb) \\
&= \delta_{ij}\mu(c^*b) \\
&=\delta_{ij} \langle[b],[c]\rangle
\end{align}
so $\pi(L_j)^*\pi(L_i)=\delta_{ij} I$.  

By definition, for $a,b,c\in\mathcal A$ we have $\langle
\pi(a)b,c\rangle := \nu(c^*ab)$.  In particular, fixing a word $w$ and
taking $a=L_w$, $b=c=I$, each state $\nu$ is a vector state in the GNS
representation:
\begin{equation}
\nu(L_w) = \langle \pi(L_w)[I],[I]\rangle.  
\end{equation}

\subsection{The GNS construction in $\mathcal S+\mathcal S^*$}\label{sec:GNS-main}  The next goal is to imitate the above construction with the NC disk algebra $\mathcal A$ replaced by its symmetric part $\mathcal S$.  The fact that $\mathcal S$ is not an algebra means the construction must be modified; it is Lemma~\ref{lem:bimodule} that makes it possible at all.

 Let $\mathcal S_0$ be the subspace of $\mathcal S$ given by
\begin{equation}
  \mathcal S_0 := {\rm span} \{ L^{(\nn)} : |\nn|\geq 1\};
\end{equation}
so that $\mathcal S={\rm span}\{I,\mathcal S_0\}$.  Let $P^2_0(\mu)$
denote the closed subspace of $P^2(\mu)$ spanned by the set
$\{[p]:p\in\mathcal S_0\}$ and $P_0:P^2(\mu)\to P^2_0(\mu)$ the
orthogonal projection.  (It is possible that $P^2_0(\mu)=P^2(\mu)$.)
\begin{prop}\label{prop:Si-gns}
Let $p\in\mathcal S_0$ be a polynomial. For each $j=1, \dots d$ the map
\begin{equation}\label{eqn:pre-S0}
 [p] \to [L_j^*p({\bf L})]
\end{equation}
is well defined, and extends to a bounded linear operator from
$P_0^2(\mu)$ to $P^2(\mu)$. Morevoer the operator ${\bf S}=(S_1,\dots S_d)$ defined by
\begin{equation}\label{eqn:S0-def}
  S_j^*[p]:= P_0[L_j^*p({\bf L})]
\end{equation}
is a row contraction on $P_0^2(\mu)$.
\end{prop}
\begin{proof}
By Lemma~\ref{lem:bimodule}, if $p\in \mathcal S_0$ then $L_i^*p({\bf
  L})\in\mathcal S$ for each $i=1,\dots n$, and again by the lemma
$q({\bf L})^*L_i^*p({\bf L})\in \mathcal S+\mathcal S^*$, so belongs
to the domain of $\mu$.  For each $i=1,\dots d$, the pairing
\begin{equation}
  ([p],[q])_i =\mu(q({\bf L})^*L_i^*p({\bf L})), \qquad p,q\in\mathcal S_0
\end{equation}
gives a well-defined, bounded bilinear form on the span of
$\{[p]:p\in\mathcal S_o\}$ in $P^2_0(\mu)$.  Indeed, since
$L_i^*L_i=I$ for each $i$, the Cauchy-Schwarz inequality for $\mu$
gives
\begin{equation}\label{eqn:cs-gns}
  |([p],[q])_i| =|\mu(q({\bf L})^*L_i^*p({\bf L}))| \leq \mu(q({\bf L})^*L_i^*L_iq({\bf L}))^{1/2}\mu(p({\bf L})^*p({\bf L}))^{1/2} =\|[p]\| \|[q]\|
\end{equation}
so $(\cdot,\cdot)_i$ is well defined and bounded (with norm at most
$1$).  Thus each of the maps (\ref{eqn:pre-S0}) is bounded, and the
operators $S_j^*$ of (\ref{eqn:S0-def}) are bounded.  To see that
${\bf S}=(S_1,\dots S_n)$ is a row contraction, we have for all
$p\in\mathcal S_0$,
\begin{align*}
  \langle \sum_{i=1}^d S_iS_i^* [p],[p]\rangle_{P^2_0(\mu)} &= \sum_{i=1}^d \langle P_0[L_i^*p],P_0[L_i^*p]\rangle_{P^2(\mu)} \\
&\leq \sum_{i=1}^d   \langle [L_i^*p],[L_i^*p]\rangle_{P^2(\mu)} \\
&= \sum_{i=1}^d   \mu(p({\bf L})^*L_iL_i^*P(L))\\
&= \mu(p({\bf L})^*p({\bf L}))\\
&= \langle [p],[p]\rangle_{P^2_0(\mu)}
\end{align*}
(Equality holds in the second-to-last line since $p\in\mathcal S_0$,
which entails $p({\bf L})=\sum_{j=1}^d L_jL_J^*p({\bf L})$).
\end{proof}

{\bf Remark:} It is very important to observe that at this point, we
cannot assert a GNS-style representation of $\mu$ in terms of
${\bf S}$; that is, the above construction does {\em not} imply that
\begin{equation}\label{eqn:bad-GNS-for-S}
\mu(L^{(\nn)})=\langle S^{(\nn)}[I],[I]\rangle_\mu  
\end{equation}
Indeed, as things stand the equation (\ref{eqn:bad-GNS-for-S}) does
not even make sense, since ${\bf S}$ is only defined on $P^2_0(\mu)$,
which need not contain $[I]$. However such a representation of $\mu$
is available when $[I]$ belongs to $P^2_0(\mu)$ (that is, when
$P^2_0(\mu)=P^2(\mu)$).  To prove this it will be helpful to consider
extensions $\nu$ of $\mu$ to the full Cuntz-Toeplitz operator system
$\mathcal A+\mathcal A^*$, and compare the GNS tuple ${\bf U}:=\pi({\bf L})$ to
${\bf S}^\mu$.  More precisely, let $\nu$ be a state on $\mathcal
A+\mathcal A^*$ and let us write $Q^2(\nu)$ for the GNS space
associated to $\nu$.  Inside $Q^2(\nu)$ there is a subspace
$Q_0^2(\nu)$ formed by taking the closed span of the elements
\begin{equation}
  \{ [L_w]:|w|\geq 1\}
\end{equation}
in $Q^2(\nu)$.  We let $Q_0$ denote the orthogonal projection onto
$Q_0^2(\nu)$.  Now, if $\mu$ is a state on $\mathcal S+\mathcal S^*$
and $\nu$ extends $\mu$, the inclusion $\mathcal S\subset \mathcal A$
induces isometric inclusions of the Hilbert spaces
\begin{equation}
  P^2(\mu)\subset Q^2(\nu), \quad P^2_0(\mu)\subset Q^2_0(\nu).
\end{equation}
Let us write ${\bf U}=(U_1, \dots ,U_d):=(\pi(L_1), \dots ,\pi(L_d)$
for the GNS tuple for $\nu$ acting in $Q^2(\mu)$. By construction the
subspace $Q^2_0(\nu)$ is invariant for the $U_j$, so we can define
${\bf V}$ to be the restriction of ${\bf U}$ to $Q_0^2(\nu)$.

We now consider the following definition:
\begin{defn}
  Let $\mu$ be a state on $\mathcal S^*+\mathcal S$ and $\nu$ be a
  state on $\mathcal A^*+\mathcal A$ extending $\mu$, and ${\bf S},
  {\bf U}$ the GNS operators associated to $\mu$ and $\nu$
  respectively.  The extension $\mu$ will be called {\em tight} if
  ${\bf V}={\bf U}|_{Q^2_0(\nu)}$ is a dilation of ${\bf S}$.  A state
  $\nu$ on $\mathcal A+\mathcal A^*$ is called {\em tight} if it is a
  tight extension of its restriction $\mu=\nu|_{\mathcal S^*+\mathcal
    S}$.
\end{defn}
In other words, starting from a state $\mu$ on the symmetric operator
system $\mathcal S+\mathcal S^*$, we have two ways of constructing row
contractions on $P^2_0(\mu)$. One is to construct the GNS tuple ${\bf
  S}$ of Proposition~\ref{prop:Si-gns}. The other is to extend the
state $\mu$ to a state $\nu$ on $\mathcal A+\mathcal A^*$, form the
GNS tuple ${\bf U}$ on $Q^2(\mu)$, then compress this tuple to
$P^2_0(\mu)\subset Q^2(\mu)$.  To call the extension $\nu$ tight is to
say these constructions coincide.  We will also see shortly that if
${\bf V}$ is a dilation of ${\bf S}$, then it is necessarily a minimal
dilation of ${\bf S}$.

At present we do not know whether or not tight extensions always
exist.  The next theorem gives a somewhat more transparent spatial
condition which characterizes tight extensions.

\begin{thm}\label{thm:tight-spatial}
  Let $\mu$ be a state on $\mathcal S+\mathcal S^*$ and $\nu$ an
  extension of $\mu$ to $\mathcal A+\mathcal A^*$.  Then $\nu$ is a
  tight extension if and only if $P_0[I]=Q_0[I]$.
\end{thm}
\begin{proof}
  Let ${\bf U}=(U_1, \dots U_d)$ be the GNS tuple for $\nu$.  By
  definition the extension is tight if and only if the restriction of
  the $U$'s to $Q^2_0(\nu)$ form a dilation of the $S$'s.  This
  happens if and only if for all polynomials $p\in \mathcal S_0$,
  \begin{equation}
    Q_0U_i^*[p] =S_i^*[p],
  \end{equation}
or more explicitly 
\begin{equation}\label{eqn:tight-spatial-key}
  Q_0[L_i^*p({\bf L})] =P_0[L_i^*p({\bf L})]
\end{equation}
Of course, (\ref{eqn:tight-spatial-key}) will always hold when
$L_i^*p({\bf L})\in\mathcal S_0$; the $p$'s with this property are the span
of the set $\{ L^{(\nn)} : |\nn|\geq 2\}$.  So what is at issue are
the cases $p({\bf L})=L_j$.  In this case, if $i\neq j$, then both sides of
(\ref{eqn:tight-spatial-key}) are $0$, while if $i=j$ we obtain the
condition $Q_0[I]=P_0[I]$.  
\end{proof}


\begin{prop}
  If $\nu$ is a tight extension of $\mu$, then ${\bf V}$ is a minimal
  dilation of ${\bf S}$.
\end{prop}
\begin{proof}
  We maintain the notation used above.  By construction $P^2_0(\mu)$
  contains the vectors $[L_1],\dots [L_n]$ (since the $L_i$ belong to
  $\mathcal S_0$), but then
\begin{equation}
  Q^2_0(\nu)\supset \bigvee_{w\in\mathbb F_+^n} U_wP^2_0(\mu) \supset  \bigvee_{w\in\mathbb F_+^n}\{U_w[L_1],\dots U_w[L_n]\} =\bigvee_{p\in\mathcal A_0} [p] =Q^2_0(\mu).
\end{equation}
In other words, the vectors $[L_i]$ are cyclic for the row isometry
${\bf U}$, but these cyclic vectors are contained in $P^2_0(\mu)$.)
This says that each containment is an equality, which gives
minimality.
\end{proof}
The point of this proposition is that it will show, for the
quasi-extreme states to be defined shortly, the GNS tuple ${\bf U}$
will be completely determined by ${\bf S}$ (as the minimal dilation of
${\bf S}$), and hence uniquely determined by $\mu$ (equivalently,
$b$). We will revisit this remark following the proof of
Theorem~\ref{thm:clark-ball}.

\begin{thm}\label{thm:tights-are-unique}
  If $\mu$ has a tight extenstion, then it is unique (that is, if
  $\nu_1$ and $\nu_2$ are tight extensions of $\mu$, then
  $\nu_1=\nu_2$).
\end{thm}
\begin{proof}
  Suppose $\nu$ is a tight extension of $\mu$ and let $w=i_1\cdots
  i_m$ be a word.  Let $\bar w= i_mi_1\dots i_{m-1}$ (remove the last
  letter of $w$ and append it at the beginning).  As shorthand write
  $i=i_m$.  Then
\begin{align}
\nu(L_w)&= \nu(L_i^*L_{\bar w}L_i)\\
&= \langle U_{\bar w}[L_i],[L_i]\rangle_{Q^2_0(\nu)} \\
&= \langle [L_i],U_{\bar w}^*[L_i]\rangle_{Q^2_0(\nu)} \\
&= \langle [L_i],S_{\bar w}^*[L_i]\rangle_{P^2_0(\mu)} 
\end{align}
which shows that $\nu(L_w)$ is completely determined by $\mu$, and hence the extension is unique.  
\end{proof}

\begin{ques} Does every state $\mu$ on $\mathcal S+\mathcal S^*$
have a tight extension to $\mathcal A+\mathcal A^*$?
\end{ques}

It is rather frustrating that this question is still
unanswered. Indeed, the proof of the foregoing theorem tells us what
the extension must be, namely
\begin{equation}
\nu(L_w):= \langle [L_i],S_{\bar w}^*[L_i]\rangle_{P^2_0(\mu)} 
\end{equation}
The difficulty is in showing that this defines a {\em positive} linear
functional.

We can now give a sufficient condition for the existence of a tight extension, in terms of the GNS space.
\begin{defn}
  A state $\mu$ on $\mathcal S+\mathcal S^*$ will be called {\em quasi-extreme} if $P^2_0(\mu)=P^2(\mu)$.
\end{defn} {\bf Remark.} The name ``quasi-extreme'' is chosen by
analogy with the one-variable case. Indeed it is an easy consequence
of the Szeg\H{o} theorem that a function $b$ is an extreme point of
the unit ball of $H^\infty(\mathbb D)$ if and only if for some
(equivalently, all) $\alpha\in \mathbb T$, one has
$P^2(\mu_\alpha)=L^2(\mu_\alpha)$.  By a standard backward-shift
argument, this latter condition is in turn equivalent to the equality
$P^2_0(\mu)=P^2(\mu)$.  So a state on $C(\mathbb T)$ (that is, a
probability measure on $\mathbb T$) is quasi-extreme by the above
definition if and only if it is an AC measure for an extreme point of
the ball of $H^\infty$.  We do not know if there is any relation
between extreme points of the unit ball and quasi-extreme states in
higher dimensions.
\begin{thm}\label{thm:extension}
Every quasi-extreme state on $\mathcal S+\mathcal S^*$ has a unique extension to a state on $\mathcal A+\mathcal A^*$, and this extension is tight.
\end{thm}
\begin{proof}
  The quasi-extremality assumption implies that the projection $P_0$
  is the identity operator, hence $[I]\in P^2_0(\mu)$, but then if
  $\nu$ is {\em any} extension, we have $[I]\in Q^2_0(\nu)$, so
  $P_0[I]=[I]=Q_0[I]$.  Thus by Theorem~\ref{thm:tight-spatial} $\nu$
  is a tight extension, but then Theorem~\ref{thm:tights-are-unique}
  gives that $\nu$ is unique.
\end{proof}

There is an operator-theoretic characterization of quasi-extremity,
using the GNS tuple ${\bf S}$:

\begin{lem}\label{lem:gns-coisom}
  The state $\mu$ is quasi-extreme if and only if its GNS tuple ${\bf
    S}=(S_1, \dots S_d)$ is co-isometric.
\end{lem}
\begin{proof}
First assume $\mu$ is quasi-extreme. It suffices to show that 
  \begin{equation}
    \sum_{j=1}^d \|S_j^*[p]\|^2 =\|[p]\|^2
  \end{equation}
for all polynomials $p\in \mathcal S_0$, since by hypothesis these vectors are dense in $P^2(\mu)$.  For this, first note that we can write 
\begin{equation}
  p({\bf L})=\sum_{j=1}^d L_j p_j({\bf L})
\end{equation}
with $p_j\in \mathcal A_0$.  Then by the orthogonality relations for the $L_i$, 
\begin{equation}
  \sum_{i=1}^d L_iL_i^*p({\bf L})=\sum_{i=1}^d \sum_{j=1}^dL_iL_i^*L_jp_j({\bf L}) = \sum_{j=1}^d L_j p_j({\bf L})= p({\bf L}).
\end{equation}
Thus, 
\begin{equation}
  \sum_{j=1}^d \|S_j^*[p]\|^2 =\sum_{j=1}^d \mu(p({\bf L})^*L_jL_j^*p({\bf L})) =\mu(p({\bf L})^*p({\bf L})) =\|[p]\|^2.
\end{equation}
For the converse, recall the proof of Proposition~\ref{prop:Si-gns},
which established (for any state $\mu$ and any $p\in \mathcal S_0$)
the inequalities
\begin{align}
    \langle \sum_{i=1}^n S_iS_i^* [p],[p]\rangle_{P^2_0(\mu)} &= \sum_{i=1}^n\langle P_0[L_i^*p],P_0[L_i^*p]\rangle_{P^2(\mu)} \\
\label{eqn:coisom-ineq} &\leq \sum_{i=1}^n   \langle [L_i^*p],[L_i^*p]\rangle_{P^2(\mu)} \\
&= \langle [p],[p]\rangle_{P^2_0(\mu)}
\end{align}
If ${\bf S}$ is coisometric, then equality holds in
(\ref{eqn:coisom-ineq}). Specializing to $p({\bf L})=L_1$, we have
$\|P_0[I]\|_\mu^2 =\|[I]\|_\mu^2 =1$, so $P_0[I]=[I]$ and hence $\mu$
is quasi-extreme.
\end{proof}
The fact that ${\bf S}$ is coisometric forces ${\bf U}$ to be a row
unitary (that is, a system of Cuntz isometries):
\begin{prop}\label{prop:GNS-unitary} If $\mu$ is a quasi-extreme state on $\mathcal S+\mathcal
  S^*$ then the GNS tuple ${\bf U}$ belonging to $\nu$ is a row unitary.
\end{prop}
\begin{proof}
  Since $\nu$ is a tight extension of $\mu$ and $\mu$ is
  quasi-extreme, it follows $Q_0[I]=P_0[I]=[I]$, and thus
  $Q_0^2(\nu)=Q^2(\nu)$. Imitating the proof of
  Lemma~\ref{lem:gns-coisom} we see that the tuple ${\bf U}$ is
  coisometric; since it is already isometric by the GNS construction,
  it is unitary.
\end{proof}

We can now prove that in the quasi-extreme case the state $\mu$ has an honest GNS representation in terms of ${\bf S}$. 
\begin{prop}\label{prop:GNS-cyclic}
  If $\mu$ is a quasi-extreme state on $\mathcal S+\mathcal S^*$, then $\mu$ is a vector state in the GNS representation, that is
  \begin{equation}\label{eqn:gns-quasi}
    \mu(p({\bf L}))=\langle p({\bf S})[I],[I]\rangle
  \end{equation}
  for all polynomials $p\in\mathcal S$. Moreover the GNS tuple ${\bf
    S}$ is cyclic, with cyclic vector $[I]$.
\end{prop}
\begin{proof}
  Let $\nu$ be the unique extension of $\mu$ to a state on $\mathcal
  A+\mathcal A^*$ coming from Theorem~\ref{thm:extension}. Since the
  extension is tight, the restricted GNS tuple ${\bf V}={\bf
    U}|_{Q^2_0(\nu)}$ for $\nu$ dilates ${\bf S}$.  Fix a polynomial
  $p\in\mathcal S$.  Then for any polynomial $q\in \mathcal S$, we
  have
  \begin{align}
    \langle [q({\bf L})],p({\bf S})[I]\rangle_\mu &= \langle p({\bf S})^*[q({\bf L})],[I]\rangle_\mu \\ & = \langle p(V)^*[q({\bf L})],[I]\rangle_\nu \\ &= \nu(p^*q) \\ &=\mu(p^*q) \\&= \langle [q({\bf L})], [p({\bf L})]\rangle.
  \end{align}
  Since this holds for all $q$, we conclude that $p({\bf
    S})[I]=[p({\bf L})]$. The identity (\ref{eqn:gns-quasi}) now
  follows by taking $q({\bf L})=[I]$.  Since the $[p({\bf L})]$ are
  dense in $P^2(\mu)=P^2_0(\mu)$ by definition, we also have that
  $[I]$ is a cyclic vector for ${\bf S}$.
\end{proof}
In the one-dimensional case, the theory of the de Branges-Rovnyak
spaces $\mathcal{H}(b)$ often splits into the extreme and non-extreme
cases.  For example, $b$ itself belongs to $\mathcal{H}(b)$ if and
only if $b$ is not extreme \cite[IV-4, V-3]{Sar-book}. It turns out
that the notion of quasi-extreme introduced above is the correct one
in this context.
\begin{defn} A contractive multiplier $b$ of $H^2_d$ will be called
  {\em quasi-extreme} if and only if the state $\mu$ representing $b$
  as in (\ref{eqn:nc-herglotz}) is quasi-extreme.
\end{defn}
\begin{thm}\label{thm:non-extreme-if-in}
  Let $b$ be a contractive multiplier of $H^2_d$. Then $b\in
  \mathcal{H}(b)$ if and only if $b$ is not quasi-extreme.
\end{thm}
\begin{proof}
  Recall the normalized NC Fantappi\`{e} transform $\mathcal V_\mu$.
  Assume $b\in \mathcal{H}(b)$ with AC state $\mu$.  Then $\mathcal
  V_\mu(x)=b$ for some $x\in P^2(\mu)$, so $\mathcal
  G_\mu(x)=\frac{b}{1-b}$.  But also $\mathcal
  G_\mu((1-\overline{b(0}))[I])=\frac{1-b(0)^*b}{1-b}$, so it follows
  that
  \begin{equation}\label{eqn:1-in-kmu}
    1= \frac{1-b(0)^*b}{1-b} -(1-b(0)^*) \frac{b}{1-b}=(1-b(0)^*)\mathcal K_\mu([I] -x);
  \end{equation}
  that is, the constant function $1$ lies in the image of $P^2(\mu)$
  under $\mathcal K_\mu$. By expanding $C_z$ in a power series and
  putting $y=(1-\overline{b(0}))([I] -x)\in P^2(\mu)$, it follows from
  (\ref{eqn:1-in-kmu}) and the definition of $\mathcal K_\mu$ that
  \begin{equation}
   1=\mathcal K_\mu(y)(z) =\sum_{\nn\in\mathbb N^d}z^{\nn} \langle y, [L^{(\nn)}]\rangle_{P^2(\mu)}.
  \end{equation}
  In other words, $y$ is orthogonal in $P^2(\mu)$ to each symmetric
  monomial $L^{(\nn)}$ with $|\nn|\geq 1$, so $y$ is a nonzero vector
  orthogonal to $P^2_0(\mu)$, which means $\mu$ is not
  quasi-extreme. Conversely, the steps of this argument reverse to
  show that if $b$ is not quasi-extreme (so that there is some nonzero
  $y\in P^2_0(\mu)^\bot\subset P^2(\mu)$), then $1$ lies in the range
  of $\mathcal K_\mu$ and hence $b\in \mathcal{H}(b)$.
\end{proof}
It is worth noting that while the proof given here works in one
variable, it is quite different from the proof in \cite{Sar-book}.
\begin{cor}
  If $b$ is quasi-extreme then so is $\alpha b$ for every unimodular
  $\alpha\in\mathbb C$.
\end{cor}
\begin{proof}
  This is immediate from Theorem~\ref{thm:non-extreme-if-in}, since
  $\mathcal{H}(b)=\mathcal H(\alpha b)$.
\end{proof}
It also follows that the family of AC states $\{\mu_\alpha\}$
associated to a given $b$ are either all quasi-extreme, or all not, a
fact which was not obvious from the definition.  Unfortunately, at
present we do not know if there is any connection between being
quasi-extreme, and being an extreme point of the set of contractive
multipliers of $H^2_d$ when $d>1$ (as noted above these notions
coincide when $d=1$).
\begin{ques}
  If $b$ is quasi-extreme, then is it an extreme point of the set of
  contractive multipliers of $H^2_d$? or vice-versa?
\end{ques}
It would be very desirable to have some other characterization of the
quasi-extreme multipliers when $d>1$.  A different characterization of
the extreme $b$ in one variable is the following: $b$ is non-extreme
if and only if $1-|b|^2$ is log-integrable, which happens if only if
there is an outer function $a\in H^\infty$ such that $|a|^2+|b|^2=1$
on $\mathbb T$. This is in turn equivalent to saying that there is an
$a$ satisfying the operator identity
\begin{equation}\label{eqn:non-extreme-a}
  M_a^*M_a +M_b^*M_b =I.
\end{equation}
However, this identity can never hold between multipliers of $H^2_d$
when $d>1$, unless $a$ and $b$ are both constant \cite[Theorem
2.3]{GuoEtal}.

\subsection{Examples}\label{eg:quasiextreme} At present we do not know any function-theoretic
characterization of the quasi-extreme $b$ when $d>1$, but it is
possible to give a few examples (and non-examples).  As noted in the
introduction, if $\mu$ is a positive measure on $\partial \mathbb
B^d$, and $b$ is given by the formula
\begin{equation}
  \frac{1+b(z)}{1-b(z)} = \int_{\partial\mathbb B^d} \frac{1+z\zeta^*}{1-z\zeta^*}\, d\mu(\zeta)
\end{equation}
then $b$ is a contractive multiplier of $H^2_d$, though not every such
$b$ is representable in this form.  Every such measure of course gives
rise to a unique state $\widetilde{\mu}$ on $\mathcal S+\mathcal S^*$
representing $b$ as in (\ref{eqn:nc-herglotz}), and by comparing
Taylor coefficients one finds that
\begin{equation}
  \widetilde{\mu}(L^{(\nn)}) = \int_{\partial
\mathbb B^d} \zeta^{\nn}\, d\mu(\zeta).
\end{equation}
In particular if we take $\mu$ to be the point mass at a 
fixed
$\zeta\in\partial \mathbb B^d$, the resulting state on $\mathcal
S+\mathcal S^*$ is called the {\em Cuntz state} $\omega_\zeta$.  The
corresponding $b$ is $b(z)=\langle z,\zeta\rangle$ and it is easy to
see this $b$ is quasi-extreme, since $[I]=\left[\sum_{j=1}^d \zeta_j^*
  L_j\right]$ in $P^2_0(\omega_\zeta)$.  (Indeed $\|[I]-[\sum
\zeta_j^*L_j]\|_{\omega_\zeta}^2 = 2-2\text{Re}\omega_\zeta(\sum
\zeta_j^*L_j) = 0$.) We will see later that all of the $\mathcal H(b)$
spaces are infinite dimensional, which gives another indication that
the classical measure $\mu$ is inadequate for our purposes---in this
example, $L^2(\mu)$ is of course one-dimensional so there can be no
identification of $L^2(\mu)$ with $\mathcal H(b)$.

If in the above construction we take $\mu$ to be a measure supported
on the circle $z_2=\cdots z_d=0$, then the resulting $b$ is a function
of $z_1$ alone, and any $b(z)=b(z_1)$ can equal any function in the
unit ball of $H^\infty(\mathbb D)$.  In this case $b$ will be
quasi-extreme if and only if $b(z_1)$ is an extreme point of the unit
ball of $H^\infty$.

A more sophisticated example, in this case for $d=2$, comes by
considering the state $\mu= \frac12 (\omega_{e_1}+\omega_{e2})$ on
$\mathcal S+\mathcal S^*$. The resulting $b$ is
\begin{equation}
  b(z_1, z_2) = \frac{z_1+z_2 -z_1z_2}{2-z_1-z_2}
\end{equation}
It is now less obvious, but this $b$ is quasi-extreme; this follows from the fact that for the polynomial
\begin{equation}
  p({\bf L}) = \frac{1}{\sqrt{6}} \left(\sum_j L_j -\sum_{j,k} L_jL_k +\sum_{j,k,l}L_jL_kL_l\right)
\end{equation}
one may verify that $\|[I]-[p({\bf L})]\|_\mu^2 =0$.  

In the other direction, if $b_1, \dots b_d$ are functions in
$H^\infty(\mathbb D)$ and each is not an extreme point, then the
product
\begin{equation}
  b(z)=b_1(z_1)\cdots b_d(z_d)
\end{equation}
is not quasi-extreme.
\section{Canonical functional models and the Gleason problem in $\mathcal H(b)$}\label{sec:gleason}

The goal of this section is to establish the uniqueness of the
contractive solution to the Gleason problem in $\mathcal{H}(b)$ when
$b$ is quasi-extreme, and study some of its properties.  In the next
section we will show that this solution admits rank-one coisometric
perturbations. If $f$ is a holomorphic function in $\mathbb B^d$, we
say that a $d$-tuple of holomorphic functions $f_1, \dots f_d$ {\em
  solves the Gleason problem for $f$} if
\begin{equation}
  f(z)-f(0) =\sum_{j=1}^d z_j f_j(z).
\end{equation}
Similarly, a $d$-tuple of linear operators $A_1, \dots A_d$ is said to solve the Gleason problem in a holomorphic space $H$ if 
\begin{equation}
  f(z)-f(0) =\sum_{j=1}^d z_j(A_jf)(z)
\end{equation}
for all $f\in H$. 

Notice that it one variable, it is trivial that the Gleason problem
for $f$ has a unique solution, given by the backward shift $f\to
(f(z)-f(0))/z$. Likewise the backward shift is the only operator
solving the Gleason problem in a holomorphic space $H$, so questions
about it focus on boundedness, etc. In contrast, in the multivariable
setting solutions to the Gleason problem for a given $f$ are never
unique, so the goal is to establish existence (and perhaps uniqueness)
of solutions satisfying some additional conditions, typically
membership in some space of functions.  It was proved by Ball and
Bolotnikov \cite{BB} that contractive solutions to the Gleason problem
in $\mathcal H(b)$ always exist. In this section we study some of
these solutions in more detail. We prove that every such solution can
be split into a sum of two operators; these being a rank-one operator
and the adjoint of a multiplication operator (each is possibly
unbounded). This structure result will be applied to obtain a
Clark-type theorem on rank-one perturbations, and to characterize the
$z$-invariant $\mathcal H(b)$ spaces.

\subsection{Functional models} In this subsection we recall a result
of Ball and Bolotnikov \cite{BB} on solutions to the Gleason problem
in the $\mathcal H(b)$ spaces.  We begin with their definition of a
{\em canonical functional model realization}.
\begin{defn}
  Given a multiplier $b$, say that the block operator matrix 
  \begin{equation}
    {\mathbf U}=\begin{pmatrix}
      A & B \\
      C & D \end{pmatrix} : \mathcal H(b)\oplus \mathbb C \to \mathcal H(b)^d \oplus \mathbb C
  \end{equation}
is a {\em canonical functional model realization} for $b$ if the following conditions are satisfied:
\begin{itemize}
\item[(1)] ${\mathbf U}$ is contractive,
\item[(2)] the $d$-tuple $A:\mathcal H(b)\to \mathcal H(b)^d$ solves the Gleason problem for $\mathcal H(b)$,
\item[(3)] $B:\mathbb C\to \mathcal H(b)^d$ solves the Gleason problem for $b$,
\item[(4)] the operators $C: \mathcal H(b)\to \mathbb C$ and $D:\mathbb C\to \mathbb C$ are given by
  \begin{equation}
    C:f\to f(0), \qquad D:\lambda\to b(0)\lambda
  \end{equation}
respectively, for all $f\in\mathcal H(b)$ and all $\lambda\in\mathbb C$.
\end{itemize}
\end{defn}
To say that ${\mathbf U}$ is a {\em realization} of $b$ means that for all $z\in\mathbb B^n$ 
\begin{equation}
  b(z)= D+ C\left(I-\sum_{j=1}^d z_jA_j\right)^{-1} \left(\sum_{j=1}^d z_jb_j\right).
\end{equation}
where we have written $B$ as a column vector $(b_1, \dots b_d)^T$ with $b_j\in\mathcal H(b)$.  The fact that ${\mathbf U}$ is contractive then entails
\begin{equation}
  B^*B \leq 1-D^*D, 
\end{equation}
or 
\begin{equation}
  \sum_{j=1}^n \|b_j\|^2_{\mathcal H(b)}\leq 1- |b(0)|^2.
\end{equation}
Moreover, since $C$ can be expressed as $C:f\to \langle f,k^b_0\rangle_{\mathcal H(b)}$, we can write $C^*C=k^b_0\otimes k^b_0$, and contractivity also entails
\begin{equation}
  A^*A \leq I_{\mathcal H(b)} -C^*C,
\end{equation}
or
\begin{equation}\label{eqn:contractive-gleason-def}
  \sum_{j=1}^d A_j^*A_j \leq I- k^b_0\otimes k^b_0.
\end{equation}
Following \cite{BB}, when (\ref{eqn:contractive-gleason-def} holds we
say $A$ is a {\em contractive solution to the Gleason problem in
  $\mathcal H(b)$}. One of the main results of \cite{BB} is
\begin{thm}\label{thm:bb-cfm}
  For each contractive multiplier $b$, there exists a canonical
  functional model realization ${\mathbf U}$. In particular, for each
  $b$ there exists a contractive solution to the Gleason problem in
  $\mathcal H(b)$, and there exist functions $b_1, \dots
  b_d\in\mathcal H(b)$ satisfying
\begin{enumerate}
\item[(i)] $\displaystyle{b(z)-b(0)=\sum_{j=1}^d z_jb_j(z)}$, and
\item[(ii)] $\displaystyle{\sum_{j=1}^d \|b_j\|^2_{\mathcal{H}(b)} \leq 1-|b(0)|^2}$
\end{enumerate}
\end{thm}
The functions $b_j$ play the same role in the present development as
$S^*b$ does in the one-variable case; in particular (ii) is the
``inequality for difference quotients'' in this setting.  

We record one more result from \cite{BB}, namely that the reproducing
kernel for $\mathcal H(b)$ can be expressed in terms of a functional
model. In particular we have for all $f\in\mathcal H(b)$
\begin{equation}
  f(z)= C(I-z{\bf A})^{-1}f
\end{equation}
or more explicitly
\begin{equation}
  f(z)=\langle (I-z{\bf A})^{-1} f, k^b_0\rangle
\end{equation}
and thus
\begin{equation}\label{eqn:func-model-kernel}
  k^b_z = (I-{\bf A}^*z^*)^{-1}k^b_0.
\end{equation}
\subsection{Unique contractive solutions in the quasi-extreme case}
\begin{defn} A contractive solution ${\bf X}=(X_1,\dots X_d)$ to the Gleason problem in $\mathcal H(b)$ will be called {\em extremal} if
  \begin{equation}\label{eqn:extremal-gleason-def}
    \sum_{j=1}^d X_j^* X_j =I-k^b_0\otimes k^b_0.
  \end{equation}
\end{defn}
(That is, equality holds in (\ref{eqn:contractive-gleason-def}).)  The
next theorem characterizes contractive and extremal solutions by their
action on reproducing kernels. To avoid trivialities we assume $b$ is
non-constant.
\begin{thm}\label{thm:canon-on-kernels}
  A $d$-tuple ${\mathbf X}=(X_1, \dots X_d)$ is a contractive solution
  to the Gleason problem in $\mathcal{H}(b)$ if and only if it acts on
  the reproducing kernels $k^b_w$ by the formula
  \begin{equation}\label{eqn:canon-on-kernels}
    (X_jk^b_w)(z) =w_j^* k^b_w(z)-b_j(z)b(w)^*.
  \end{equation}
  where $b_1,\dots b_d$ are functions in $\mathcal{H}(b)$ satisfying
\begin{enumerate}
\item[(i)] $\displaystyle{b(z)-b(0)=\sum_{j=1}^d z_jb_j(z)}$, and
\item[(ii)] $\displaystyle{\sum_{j=1}^d \|b_j\|^2_{\mathcal{H}(b)} \leq 1-|b(0)|^2}$
\end{enumerate}
The solution is extremal if and only if equality holds in (ii).
\end{thm}
\begin{proof}
  First, suppose $b_j$ exist satisfying (i) and (ii), and the $X_j$
  are defined by (\ref{eqn:canon-on-kernels}).  Then for all $z,
  w\in\mathbb B^d$,
  \begin{align}
    \sum_{j=1}^d z_j(X_jk^b_w)(z) &=\sum_{j=1}^d z_j\left(w_j^*  k^b_w(z)-z_jb_j(z)b(w)^*\right) \\
&= {zw^*} \frac{1-b(z)b(w)^*}{1-{zw^*}} - (b(z)-b(0))b(w)^* \\
&= {zw^*}\frac{1-b(z)b(w)^*}{1-{zw^*}} +(1-b(z)b(w)^*) -(1-b(0)b(w)^*)  \\
&= k^b_w(z)-k^b_w(0) 
  \end{align}
  where we have used property (i) of the $b_j$.  Thus the $X_j$ solve
  the Gleason problem on the linear span of the $k^b_w$. Once we show
  that $\sum X_j^*X_j\leq I-k^b_0\otimes k^b_0$ on this span, it follows
  that the $X_j$ are bounded and have unique bounded extensions to
  $\mathcal{H}(b)$; a routine approximation argument shows that these
  extensions solve the Gleason problem for all of $\mathcal{H}(b)$.
  So, compute:
\begin{align}
  \langle X_jk^b_w, X_jk^b_z\rangle &= \langle w_j^*  k^b_w-b_jb(w)^*, z_j^* k^b_z-b_jb(z)^*\rangle \\
&= z_jw_j^*  k(z,w) -z_jb_j(z)b(w)^* -w_j^* b_j(w)^* b(z) + \|b_j\|^2b(z)b(w)^*.
\end{align}
Summing over $j$, and using properties (i) and (ii) of the $b_j$, we have
\begin{align}
\label{eqn:firstline}  \langle \sum_{j=1}^d X_j^*X_j k^b_w,k^b_z\rangle &\leq {zw^*} k(z,w) -(b(z)-b(0))b(w)^* -b(z)(\overline{b(w)-b(0)}) +(1-|b(0)|^2)b(z)b(w)^*\\
&=  {zw^*} k(z,w) +(1-b(z)b(w)^*) -(1-b(z)b(0)^*)(1-b(0)b(w)^*) \\
&= \langle k^b_w, k^b_z\rangle -\langle k^b_w,k^b_0\rangle \langle k^b_0,k^b_z\rangle \\
&= \langle (I-k^b_0\otimes k^b_0)k^b_w,k^b_z\rangle
\end{align}
This shows $\sum X_j^*X_j \leq I-k^b_0\otimes k^b_0$, with equality if and only if equality holds in (\ref{eqn:firstline}), if and only if equality holds in (ii).

Conversely, suppose $X_j$ are operators on $\mathcal{H}(b)$ which solve the Gleason problem and satisfy $\sum X_j^*X_j\leq I-k^b_0\otimes k^b_0$.  For each $j=1, \dots d$ and each $w\in \mathbb B^d$, define a function $f_{j,w}\in \mathcal{H}(b)$ by the formula
\begin{equation}\label{eqn:form-of-T}
  f_{j,w}(z) =w_j^* k^b_w(z) - (X_jk^b_w)(z). 
\end{equation}
We must show $f_{j,w}=b_j b(w)^*$ for some $b_j$ satisfying (i) and (ii).  A computation similar to the verification in the first part of the proof shows that, since the $X_j$ are assumed to solve the Gleason problem, we must have
\begin{equation}\label{eqn:fj-temp}
  \sum_{j=1}^d z_j f_{j,w}(z) =(b(z)-b(0))b(w)^*.
\end{equation}
Using this identity, and again imitating the algebra in the first part
of the proof, the hypothesis $\sum X_j^*X_j\leq I-k^b_0\otimes k^b_0$
entails the kernel inequalities
\begin{align}
  k(z,w)-k^b_0(z)k^b_0(w)^* &\geq \sum_{j=1}^d \langle X_jk^b_w, X_jk^b_z\rangle \\
&= {zw^*} k(z,w) -2b(z)b(w)^* +\sum_{j=1}^d \langle f_{j,w},f_{j,z}\rangle.
\end{align}
This simplifies to 
\begin{equation}\label{eqn:f-condition}
  \sum_{j=1}^d \langle f_{j,w},f_{j,z}\rangle \leq (1-|b(0)|^2)b(z)b(w)^*.
\end{equation}
This inequality implies, via Douglas's factorization lemma, that there
is a contractive linear map from $\mathbb C$ to the direct sum of $d$
copies of $\mathcal{H}(b)$ taking $b(w)^*$ to the column vector whose
$j^{th}$ entry is $(1-|b(0)|^2)^{-1/2}f_{j,w}$. Such a map must send
the scalar $1$ to a vector in the unit ball of $\mathcal{H}(b)^d$.  If
we write $a_j$ for the $j^{th}$ entry of this vector, then $\sum
\|a_j\|^2\leq 1$, and we have
\begin{equation}
 (1-|b(0|^2)^{-1/2} f_{j,w}=a_jb(w)^*.
\end{equation}
Rescale: put $b_j=(1-|b(0)|^2)^{1/2}a_j$; then $\sum\|b_j\|^2 \leq 1-|b(0)|^2$ and
\begin{equation}
  f_{j,w}=b_jb(w)^*.
\end{equation}
But now from (\ref{eqn:fj-temp}) we have for all $z,w\in\mathbb B^d$ 
\begin{equation}
  \sum_{j=1}^d z_jb_j(z)b(w)^* =(b(z)-b(0))b(w)^*.
\end{equation}
Since $b$ is not identically $0$, we conclude $\sum z_jb_j(z)=b(z)-b(0)$, so the $b_j$ satisfy (i) and (ii), and from (\ref{eqn:form-of-T}), the $X_j$ have the claimed form.  In the case of equality $\sum_{j=1}^d X_j^*X_j =I-k^b_0\otimes k^b_0$,, we have also equality in (\ref{eqn:f-condition}), and it is a straightforward matter to verify that this propagates through the calculation to give equality in (ii) (in this case the contractive linear map which produces the $b_j$ is isometric).
\end{proof}
The above theorem also lets us obtain a formula for the action of the
$X_j^*$ on arbitrary elements of $\mathcal H(b)$.  Indeed using Theorem~\ref{thm:canon-on-kernels} and the relation
  \begin{equation*}
    X_j^*f(z) = \langle  X^*f, k^b_z\rangle =\langle f, X_jk^b_z\rangle
  \end{equation*}
we have for all $f\in \mathcal H(b)$
\begin{equation}\label{eqn:X*-formula}
  X_j^*f(z)= z_jf(z) -\langle f, b_j\rangle_{\mathcal H(b)} b(z).
\end{equation}
This makes the next corollary almost immediate.
\begin{cor}\label{cor:z-invariant} Let $b$ be a contractive multiplier of $\mathcal H^2_d$. Then the following are equivalent:
  \begin{itemize}
  \item[i)] $z_j \mathcal H(b)\subset \mathcal H(b)$ for all $j=1, \dots d$
  \item[ii)] $b\in \mathcal H(b)$
  \item[iii)] $b$ is not quasi-extreme.
  \end{itemize}
\end{cor}
\begin{proof}
  The equivalence of (ii) and (iii) is
  Theorem~\ref{thm:non-extreme-if-in}.  For the equivalence of (i) and
  (ii), first recall that by Theorem~\ref{thm:bb-cfm}, for any $b$
  there exists a contractive solution to the Gleason problem $(X_1,
  \dots X_d)$.  From the formula (\ref{eqn:X*-formula} we see that if
  $b\in \mathcal H(b)$, then $z_jf\in \mathcal H(b)$ for all $f\in
  \mathcal H(b)$ and all $j$. Conversely, if $\mathcal H(b)$ is
  $z_j$-invariant for each $j$, since $b$ is non-constant we can
  choose $j$ such that $b_j\neq 0$.  Then specializing to $f=b_j$ we
  have
\begin{equation}
  \|b_j\|^2 b(z)=z_jb_j(z) -(X_j^*b_j)(z).
\end{equation}
and the right side lies in $\mathcal H(b)$ by hypothesis. 
\end{proof}
{\bf Remark:} Note that it is possible that $\mathcal H(b)$ is $z_j$-invariant for some $j$ but not others. A simple example in two variables is $b(z_1, z_2)=z_1$: then we can take $b_1(z)=1, b_2(z)=0$.  This $b$ is associated to the Cuntz state $\omega_{e_1}$ and is quasi-extreme (Example~\ref{eg:quasiextreme}). However from (\ref{eqn:X*-formula}) we see that $\mathcal H(b)$ is invariant for $z_2$ but not $z_1$. 

We also observe that once $b\in \mathcal H(b)$, then $\mathcal H(b)$
also contains the constant functions, and therefore, by
$z$-invariance, all polynomials. In one variable the polynomials are
dense in $\mathcal H(b)$ when $b$ is non-extreme, but so far we have
been unable to prove this when $d>1$ (though it seems very likely to
be true).

\begin{ques} If $b$ is non-extreme, are the polynomials dense in
  $\mathcal H(b)$?
\end{ques}

We are now in a position to prove the uniqueness of the contractive
solution (which will in fact be extremal) to the Gleason problem when
$b$ is quasi-extreme. By Theorem~\ref{thm:canon-on-kernels}, it
suffices to produce the functions $b_j$. The next lemma shows how to
do this, starting from the AC state for $b$. We make use of the
normalized NC Fantappi\`{e} transform $\mathcal V_\mu$ and the GNS
tuple ${\bf S}$ associated to the state $\mu$.

\begin{lem}\label{lem:gleason-for-phi}
  Let $b$ be a quasi-extremal multiplier, with AC state $\mu$.  Then the functions
  \begin{equation}
    b_j:= (1-b(0)^*)\mathcal V_\mu(S_j^*[I])
  \end{equation}
belong to $\mathcal{H}(b)$, satisfy $\sum_j \|b_j\|_{\mathcal{H}(b)}^2 =1-|b(0)|^2$, and solve the Gleason problem for $b$; that is
\begin{equation}
  b(z)-b(0)=\sum_{j=1}^d z_jb_j(z).
\end{equation}
\end{lem}
\begin{proof}
From (\ref{eqn:NC-db-kernel-2}) we have
\begin{equation}
 \mu((I-\zlstar)^{-1})(z) =\frac{1}{1-b(0)^*} \frac{1-b(0)^*b(z)}{1-b(z)}
\end{equation}
and
\begin{equation}\label{eqn:mu-of-I}
  \mu(I)=\frac{1-|b(0)|^2}{|1-b(0)|^2}.
\end{equation}
It follows that
\begin{align}
\mu(\sum_{k=1}^\infty (\zlstar)^k ) &=   \mu((I-\zlstar)^{-1}) -\mu(I))\\
&= \label{eqn:muCz}\frac{1}{1-b(0)^*}\frac{b(z)-b(0)}{1-b(z)} 
\end{align}
By the assumption that $b$ is quasi-extreme, there is a sequence of polynomials $p_m\in \mathcal S_0$ such that $[p_m]\to [I]$ in $P^2(\mu)$.  Then for each integer $n\geq 1$, 
\begin{align*}
  \mu((\zlstar)^n) &= \langle [I], [(\zlstar)^n]\rangle_H\\
 &= \lim_{m\to \infty} \langle [p_m],[(\zlstar)^n]\rangle \\
&=\lim_{m\to \infty} \mu((\zlstar)^n p_m(L))\\
&=\lim \sum_{j=1}^d z_j \mu((\zlstar)^{n-1} L_j^*p_m(L)) \\
&=\lim \sum_{j=1}^d z_j \langle S_j^*[p_m], (\zlstar)^{n-1}\rangle \\
&= \sum_{j=1}^d z_j \langle S_j^*[I], (\zlstar)^{n-1}\rangle
\end{align*}
Summing from $n=1$ and multiplying by $(1-b)$, we obtain
\begin{align}
  \frac{b(z)-b(0)}{1-b(0)^*}&=\sum_{j=1}^d z_j (1-b(z))\mathcal K_\mu(S_j^*[I]) \\
&= \sum_{j=1}^d z_j \mathcal V_\mu(S_j^*[I]).
\end{align}
Multiplying this by $1-b(0)^*$ and applying (\ref{eqn:muCz}) and the definition of $b_j$, we have
\begin{equation}
  b(z)-b(0)=\sum_{j=1}^d z_j b_j(z),
\end{equation}
so the functions $b_j$ solve the Gleason problem as claimed. They belong to $\mathcal{H}(b)$ since they lie in the range of $\mathcal V_\mu$. For the norm computation we have 
\begin{equation}
\sum_{j=1}^d \|S_j^*[I]\|^2 = \|[I]\|^2 =\mu(I) = \frac{1-|b(0)|^2}{|1-b(0)|^2}.   
\end{equation}
by Lemma \ref{lem:gns-coisom} and equation (\ref{eqn:mu-of-I}).  Thus
using the definition of $b_j$ and the fact that $\mathcal V_\mu$ is unitary,
\begin{align*}
  \sum_{j=1}^d \|b_j\|^2 &= |1-b(0)|^2 \sum_{j=1}^d \|\mathcal
  V_\mu(S_j^*[I])\|^2 \\
&= |1-b(0)|^2 \sum_{j=1}^d \|S_j^*[I]\|^2 \\
&= 1-|b(0)|^2.
\end{align*}
\end{proof}

\begin{thm}
  If $b$ is quasi-extreme, then there is a unique contractive solution to the Gleason problem in $\mathcal{H}(b)$, and this solution is extremal. 
\end{thm}
\begin{proof}
Define 
\begin{equation}
  X_jk^b_w =w_j^* k^b_w -b(w)^*b_j
\end{equation}
where the $b_j$ are chosen as in Lemma~\ref{lem:gleason-for-phi}. It
is then immediate from Theorem~\ref{thm:canon-on-kernels} that ${\bf
  X}=(X_1, \dots X_j)$ is an extremal solution to the Gleason problem
in $\mathcal H(b)$.

Uniqueness will be proved by contradiction. Suppose there are two Gleason tuples ${\bf X}, \widetilde{{\bf X}}$.  These must be defined as in equation (\ref{eqn:canon-on-kernels}), for functions $b_j$ and $\widetilde{b}_j$ satisfying conditions (i) adn (ii) of Theorem~\ref{thm:canon-on-kernels}.  But then for each $j$ the densely defined operator
  \begin{equation}
    (X_j-\widetilde{X}_j) k^b_w =(b_j-\widetilde{b}_j)b(w)^*
  \end{equation}
is bounded on $\mathcal{H}(b)$, and is nonzero for some $j$.  Fix such a $j$; put $g=b_j-\widetilde{b}_j$.  So $k^b_w\to b(w)^*g$ is a bounded rank-one operator, which means that $k^b_w\to b(w)^*$ extends to a bounded linear functional on $\mathcal{H}(b)$.  Then there is an $h\in \mathcal{H}(b)$ with
\begin{equation}
  b(w)^* =\langle k^b_w,h\rangle h(w)^*
\end{equation}
and it follows that $h=b$, so $b\in \mathcal{H}(b)$.  Since $b$ was assumed quasi-extreme, this contradicts Theorem~\ref{thm:non-extreme-if-in}.
\end{proof}

\section{Rank-one perturbations and intertwining}\label{sec:intertwining}

The goal of this section is to prove Theorem~\ref{thm:clark-ball},
which is the analog of the one-variable Theorem~\ref{thm:clark2}.  Fix
a quasi-extreme multiplier $b$ with its family of AC states
$\{\mu_\alpha\}$.  To unclutter the notation we will write $\mathcal
V_\alpha$ for the Fantappie transform $\mathcal V_{\mu_\alpha}$.  As
before, ${\bf X}$ denotes the unique solution to the Gleason
problem in $\mathcal{H}(b)$, and we write ${\bf S}^\alpha=(S_1^\alpha,
\dots S_d^\alpha)$ for the co-isometric GNS tuple acting on the GNS
space $P^2(\mu_\alpha)$.
\begin{thm}\label{thm:clark-ball}
  Let $b$ be a quasi-extreme multiplier of $H^2_d$.  Then the rank-one perturbation of ${\bf X}$ defined by
  \begin{equation}\label{eqn:perturb}
    X_j +\abar(1-\abar b(0))^{-1} b_j\otimes k^b_0
  \end{equation}
is cyclic, isometric, and unitarily equivalent to ${\bf S}^{\alpha *}$ under the normalized Fantappi\`{e} transform $\mathcal V_\alpha$:
\begin{equation}
  \mathcal V_\alpha S_j^{\alpha *} = (X_j +\abar(1-\abar b(0))^{-1} b_j\otimes k^b_0) \mathcal V_\alpha.
\end{equation}

Moreover, if $\nu_\alpha$ is the unique extension of $\mu_\alpha$ to $\mathcal A+\mathcal A^*$, then the GNS construction applied to $\nu_\alpha$ produces a Cuntz tuple ${\bf U}^\alpha$, which is unitarily equivalent to the minimal isometric dilation of ${\bf S}^\alpha$. 
\end{thm}

\begin{proof}
Since we already know ${\bf S}$ is cyclic and coisometric (Lemma~\ref{lem:gns-coisom} and Proposition~\ref{prop:GNS-cyclic}), everything follows once we prove the intertwining property; and in fact the intertwining holds even when $b$ is not quasi-extreme.

To prove the intertwining relation, recall from the proof of Theorem~\ref{thm:Vmu} that the NC kernel functions
\begin{equation}
  G_w^\alpha =(1-\alpha b(w)^*)[(1-{{\bf L}w^*})^{-1}]
\end{equation}
are dense in $P^2(\mu_\alpha)$, and $\mathcal V_\alpha$ takes $G_w^\alpha$ onto the reproducing kernel $k^b_w$ of $\mathcal{H}(b)$.  To compute the action of $S_j^{\alpha *}$ on $G_w^\alpha$, we first have for integers $n\geq 1$
\begin{align}
  S_j^{\alpha *}[({\bf L}w^*)^n] &= [L_j^*({\bf L}w^*)^n] \\
&= w_j^*  [({\bf L}w^*)^{n-1}] 
\end{align}
and, when $n=0$, from the definition of $b_j$ in Lemma~\ref{lem:gleason-for-phi} 
\begin{equation}
  S_j^{\alpha *} [I] = \abar (1-\abar b(0))^{-1}\mathcal V_\alpha^{-1} b_j.
\end{equation}
Summing over $n$, we obtain
\begin{align}
  S_j^{\alpha *} G_w^\alpha &= (1-\alpha b(w)^*)\sum_{n=0}^\infty S_j^{\alpha *} [({\bf L}w^*)^n] \\
&= (1-\alpha b(w)^*) \left( \abar (1-\abar b(0))^{-1}\mathcal V_\alpha^{-1} b_j + w_j^*  \sum_{k=0}^\infty [({\bf L}w^*)^n] \right)
\end{align}
and therefore
\begin{equation}\label{eqn:s-check}
  \mathcal V_\alpha S_j^{\alpha *} G_w^\alpha = \abar (1-\abar b(0))^{-1}(1-\alpha b(w)^*) b_j +w_j^*  k^b_w
\end{equation}
On the other hand, by the definition of $X_j$, 
\begin{align}
(X_j+  \abar(1-\abar b(0))^{-1} b_j\otimes k^b_0)\mathcal V_\alpha G_w^\alpha &= (X_j+  \abar(1-\abar b(0))^{-1} b_j\otimes k^b_0)k^b_w \\
&= w_j^* k^b_w -b_jb(w)^* + \abar (1-\abar b(0))^{-1} (1-b(0)b(w)^*) b_j \\
&= w_j^* k^b_w + \abar (1-\abar b(0))^{-1}(1-\alpha b(w)^*) b_j
\end{align}
which agrees with (\ref{eqn:s-check}).

Finally, the claims about the Cuntz tuple ${\bf U}^\alpha$ follow from the
fact that $\mu$ is quasi-extreme and Proposition~\ref{prop:GNS-unitary}.
\end{proof}
Let us recapitulate the relationship between the function $b$, the
state $\mu$ on $\mathcal S+\mathcal S^*$ representing $b$, and the
operator tuples ${\bf S}^{\alpha}$, ${\bf U}^\alpha$.  Starting with
$b$ one obtains the AC states ${\mu_\alpha}$ via the NC Herglotz
representation.  Since $b$ is quasi-extreme, each $\mu_\alpha$ is
quasi-extreme and determines a coisometric tuple ${\bf
  S}^\alpha$. This ${\bf S}^\alpha$ has a minimal row-unitary dilation
${\bf U}^\alpha$.  On the other hand, $\mu_\alpha$ has a unique
extension to a positive functional $\nu_\alpha$ on the full
Cuntz-Toeplitz operators system $\mathcal A+\mathcal A^*$.  Applying
the GNS construction to $\nu_\alpha$ gives ${\bf U}^\alpha$ again.  In
this sense we think of $\nu_\alpha$ as the ``spectral measure'' of the
row unitary ${\bf U}^\alpha$.  Moreover, a suitable rank-one
perturbation of ${\bf S}^{\alpha *}$ is unitarily equivalent, via the NC
Fantappie transform $\mathcal V_\alpha$, to the unique contractive
solution to the Gleason problem in $\mathcal H(b)$.

The only difference between this picture and the one-variable
situation is, of course, that there is no distinction between
$\mathcal S+\mathcal S^*$ and $\mathcal A+\mathcal A^*$; they are both
just (dense subspaces of) $C(\mathbb T)$, and ${\bf S}$ and ${\bf U}$ are both just the
unitary operator $M_\zeta$ acting on $P^2(\mu)=L^2(\mu)$.
  
A natural question which arises at this point is: which unitaries
${\bf U}$ can arise by this construction? In one variable the answer
is simple: every cyclic unitary operator. In the present setting, the
answer is somewhat more delicate, in that the row unitary ${\bf U}$
must not only be cyclic (thus determining a ``spectral measure''
$\nu$), but ${\bf U}$ must also be the minimal dilation of its
compression to the subspace $P^2(\mu)\subset Q^2(\nu)$.  This will be
explored further in a separate paper examining the characteristic
functions associated to rank-one perturbations of ${\bf S}$ and ${\bf
  U}$.

\section{Spectral results}\label{sec:spectra}
Finally, we examine the spectra of the solutions ${\bf X}$ to the
Gleason problem and the GNS tuples ${\bf S}$.  We begin with some
preliminaries on angular derivaties in the ball, in particular for
multipliers of $H^2_d$.  

Say that a point $\zeta\in\partial \mathbb B^d$ is a {\em C-point} for $b$ if
\begin{equation}\label{eqn:cpoint-def}
  \liminf_{z\to \zeta} \frac{1-|b(z)|^2}{1-|z|^2} = L <\infty.
\end{equation}
By \cite[Section 8.5]{Rud1}, $\zeta$ is a C-point for $b$ if and only
if $b$ and its directional derivative $D_\zeta b$ both have finite
limits as $z\to \zeta$ non-tangentially, with $\lim_{z\to
  \zeta}|b(z)|=1$ and $\lim_{z\to \zeta} D_\zeta b(z)>0$ (briefly, $b$
has a {\em finite angular derivative} at $\zeta$). However when the
function $b$ is a contractive multiplier of $H^2_d$, a somewhat
stronger theorem is available (see \cite{Jur-jc}). In particular there
is a connection between angular derivatives of $b$ and the $\mathcal
H(b)$ spaces which closely parallels the one-dimensional results of
Sarason \cite[Chapter VI]{Sar-book}.

We summarize the results needed from \cite{Jur-jc} in the following theorem:
\begin{thm}\label{thm:JC}
Let $b$ be a contractive multiplier of the Drury-Arveson space
$H^2_d$ and let $\zeta\in \partial \mathbb B^d$. The following are
equivalent:
\begin{itemize}
\item[i)] $\zeta$ is a C-point for $b$
\item[ii)] there exists
$\alpha\in\mathbb T$ such that the function
\begin{equation}
k^b_\zeta(z):= \frac{1-b(z)\alpha^*}{1-z\zeta^*}  
\end{equation}
belongs to $\mathcal H(b)$.
\end{itemize}
When these occur, $b$ has nontangential limit
$\alpha$ at $\zeta$, and additionally every $f\in \mathcal H(b)$ has a
finite nontangential limit at $\zeta$, equal to $\langle
f,k^b_\zeta\rangle_{\mathcal H(b)}$. Moreover we have $\|k^b_\zeta\|^2 =L$.
\end{thm}

It what follows we will abuse the notation slightly and write $S_j^*$
for the rank-one perturbation of $X_j$ in (\ref{eqn:perturb}).
\begin{thm}\label{thm:eigenvalues}
  Let $b$ be quasi-extreme with $AC$ states $\{\mu_\alpha\}_{\alpha\in\mathbb T}$ and ${\bf
    S^\alpha}$ the Clark tuple for $\mu_\alpha$.  For
  fixed $\zeta\in\partial \mathbb B^d$, the eigenvalue problem
  \begin{equation}
    \label{eqn:eigenvalues of S}
    \sum_{j=1}^d \zeta_j^*S_j^\alpha h = h
  \end{equation}
has a solution in $\mathcal H(b)$ if and only if $b$ has finite angular derivative at
$\zeta$ and $b(\zeta)=\alpha$, in which case the eigenspace is
one-dimensional and spanned by $k^b_\zeta$. 
\end{thm}
\begin{proof}
  First assume (\ref{eqn:eigenvalues of S}) has a nonzero solution
  $h\in \mathcal H(b)$.  Write $\widetilde{b}(z)=\sum_{j=1}^d \zeta_j
  b_j(z)$. Then using (\ref{eqn:perturb}) and (\ref{eqn:X*-formula})
  to compute $S_j^\alpha$, we have
  \begin{align}
    h(z) & = \left(\sum_{j=1}^d \zeta_j^* S_j^\alpha h\right)(z)\\ &= z\zeta^*
    h(z) -\langle h, \widetilde{b}\rangle_{\mathcal H(b)} b(z) +\alpha\langle h,
    \widetilde{b}\rangle_{\mathcal H(b)} k^b_0
  \end{align}
and solving for $h$ we find
\begin{equation}
  \label{eq:1}
  h(z) = \frac{\alpha \langle h, \widetilde{b}\rangle}{1-b(0)^*\alpha}
  \frac{1-b(z)\alpha^*}{1-\langle z,\zeta\rangle} = ck^b_\zeta(z)
\end{equation}
for some nonzero $c$. Thus by Theorem~\ref{thm:JC}, $b$ has an angular
derivative at $\zeta$ with $b(\zeta)=\alpha$.  Conversely, suppose the angular derivative condition
holds at $\zeta$, with $b(\zeta)=\alpha$. Then by Theorem~\ref{thm:JC}
the function $k^b_\zeta$ lies in $\mathcal H(b)$. Note also that by
the reproducing property of $k^b_\zeta$ at $\zeta$, we have
\begin{equation}
  \label{eq:2}
  \langle k^b_\zeta, \widetilde b\rangle = \sum_{j=1}^d
  \zeta_j^*b_j(\zeta)^* = b(\zeta)^*-b(0)^* = \alpha^*-b(0)^*.
\end{equation}
With this in hand, repeating the calculation in the first part of the proof shows that
\begin{equation}
\sum_{j=1}^d \zeta_j^* S_j^\alpha k^b_\zeta = k^b_\zeta.    
\end{equation}

\end{proof}

Finally, we include a result on the essential Taylor spectrum of ${\bf
  X}$. For this result we do not need to assume $b$ is quasi-extreme,
and ${\bf X}$ can be any contractive solution to the Gleason problem
in $\mathcal H(b)$.  First let us note that while the operators $X_j$
do not commute, we see from Theorem~\ref{thm:canon-on-kernels} that
the commutators $[X_i, X_j]$ have finite rank. Thus if we let $\pi$
denote the quotient map to the Calkin algebra, then the $\pi(X_j)$
form a commuting row contraction, and it then makes sense to talk
about its Taylor spectrum.  It turns out that we do not need the
definition of the Taylor spectrum in the proof of the next theorem,
only the fact that the spectral mapping theorem holds for it (and even
this we need only for polynomial mappings; which means that
Theorem~\ref{thm:taylor-spectrum} is valid for the Harte spectrum as
well). That is, if $\sigma(T_1, \dots T_d)$ denotes the Taylor
spectrum of a commuting $d$-tuple of operators $T_1, \dots T_d$, then
for any analytic polynomial $p$ in $d$ variables we have
\begin{equation}\label{eqn:spectral-mapping}
  p(\sigma(T_1, \dots T_d)) =\sigma(p(T_1, \dots T_d)).
\end{equation}
\begin{thm}\label{thm:taylor-spectrum}
  Let $X$ be a contractive solution to the Gleason problem in $\mathcal H(b)$.  Then the Taylor spectrum of $\pi(X)$ contains the unit sphere $\partial \mathbb B^d$.
\end{thm}
In one variable, Sarason proves in \cite[Theorem
V-8]{Sar-book} that an open arc $I\subset \mathbb T$ lies in the
resolvent set of $X^*$ if and only if every function in $\mathcal
H(b)$ can be analytically continued across $I$. In higher dimensions,
our result says that this is still true, though in a vacuous way: the
spectrum of $\pi({\bf X})$ contains the entire sphere, and it will
turn out that there is no open set of $\partial \mathbb B^d$ across which all $f\in
\mathcal H(b)$ can be continued.

We begin with two lemmas; it is the second lemma that does most of the work. 
\begin{lem}\label{lem:resolvent-mapping}
  A point $\zeta\in\mathbb B^d$ belongs to the Taylor spectrum of
  $(T_1, \dots T_d)$ if and only if $(I-{\bf T}\zeta^*)$ is not invertible. 
\end{lem}
\begin{proof}
  This follows immediately from the spectral mapping property (\ref{eqn:spectral-mapping}) applied to $T$ and the polynomial $p(z)=1-z\zeta^*$.  
\end{proof}

\begin{lem}\label{lem:closed-range-cpoint}
  Let ${\bf X}$ be any contractive solution to the Gleason problem in
  $\mathcal H(b)$ and let $\zeta\in\mathbb B^d$.  If $I-\zeta{\bf X}^*$ has closed range, then $\zeta$ is a C-point for $b$. 
\end{lem}
\begin{proof}
  Notice that the quantity in the definition of C-point
  (\ref{eqn:cpoint-def}) is nothing but $\|k^b_z\|^2$.  Now from the
  expression for the reproducing kernel in terms of ${\bf X}$, we have
  \begin{equation}\label{eqn:kernel-resolvent}
    k^b_z = (I-z^*{\bf X}^*)^{-1}k^b_0.
  \end{equation}
First we show that if $I-\zeta{\bf X}^*$ has closed range, then its
range contains $k^b_0$.  For this it suffices to show that $k^b_0$ is
always orthogonal to the kernel of $I-\zeta^* {\bf X}$, or what is the
same, that if $f\in\ker(I-\zeta^*{\bf X})$, then $f(0)=0$.  To see this, for such $f$ we have
\begin{equation}\label{eqn:f-ker-first}
  f(z) = \sum_{j=1}^d \zeta_j^* (X_jf)(z)
\end{equation}
so in particular
\begin{equation}
  f(0)=\sum_{j=1}^d \zeta_j^* (X_jf)(0).
\end{equation}
Now apply $\zeta_k^*X_k$ to (\ref{eqn:f-ker-first}), sum over $k$, and evaluate at $z=0$.  We get
\begin{equation}
  f(0)=\sum_{k=1}^d \zeta_k^* X_kf)(0) = \sum_{j,k=1}^d \zeta_k^*\zeta_j^* (X_kX_jf)(0).
\end{equation}
Continuing in this manner, we see that for each integer $m\geq 0$ we have
\begin{equation}
  f(0) =\sum_{|\nn|=m} \zeta^{\nn *} (X^{(\nn)}f)(0).
\end{equation}
Using the Taylor expansion for $f$ in terms of the $X$'s, we conclude that for this $\zeta$ and all $0\leq r<1$,
\begin{equation}\label{eqn:growth-of-f}
  f(r\zeta) =\sum_{m=0}^\infty r^m \sum_{|\nn|=m} \zeta^{\nn *} (X^{(\nn)}f)(0) =(1-r)^{-1} f(0).
\end{equation}
But $f$ belongs to $\mathcal H(b)$ and hence also to $H^2_d$, so it must satisfy the estimate
\begin{equation}
  |f(z)| = o((1-|z|)^{-1}) \quad \text{as }|z|\to 1.
\end{equation}
This is only possible in (\ref{eqn:growth-of-f}) if $f(0)=0$.

So, assuming $(I-\zeta^*X)$ has closed range, we conclude that there exists a function $h\in\mathcal H(b)$ so that $k^b_0=(I-\zeta{\bf X}^*)h$. Substitute this into the expression (\ref{eqn:kernel-resolvent}), and let $z=r\zeta$ for $r<1$.  Then
\begin{equation}\label{eqn:closed-range-lemma-cpoint}
  k^b_z =(I-r\zeta{\bf X}^*)^{-1} (I-\zeta{\bf X}^*)h.
\end{equation}
Now if $T$ is any contractive operator, one easily checks that
\begin{equation}
  (I-rT)^{-1}(I-T) = I-(1-r)(I-rT)^{-1}T,
\end{equation}
and that $\|(I-rT)^{-1} \|=O((1-r)^{-1})$.  Applying this to $T=\zeta{\bf X}^*$, we see from (\ref{eqn:closed-range-lemma-cpoint}) that $\|k^b_z\|$ stays bounded as $z\to \zeta$ along a radius, and hence $\zeta$ is a C-point for $b$.
\end{proof}
The last ingredient we need is the following result on the boundary behavior of bounded analytic functions in the ball, due to Rudin \cite[Theorem 1.2]{Rud2}.
\begin{thm}\label{thm:rudin}
  Suppose that 
  \begin{itemize}
  \item $\Gamma$ is a nonempty open set in $\partial \mathbb B^d$, 
  \item $r_j$ increases to $1$ as $j\to \infty$, 
  \item $f$ is a nonconstant holomorphic function bounded by $1$ in $\mathbb B^d$, and $\lim_{r\to 1}|f(r\zeta)|=1$ for a.e. $\zeta\in\Gamma$.
  \end{itemize}
Then $\Gamma$ has a dense $G_\delta$ subset $H$ such that the set
\begin{equation}
  \{f(r_j\zeta):j=1,2,3\dots\}
\end{equation}
is dense in the unit disk for every $\zeta\in H$.  
\end{thm}
In particular, under the conditions of this theorem we see that 
\begin{equation}\label{eqn:derivs-blow-up}
  \limsup_{r\to 1}|(D_\zeta f)(r\zeta)|=+\infty \quad \text{for every }\zeta\in H.
\end{equation}
\begin{proof}[Proof of Theorem~\ref{thm:taylor-spectrum}]
  We suppose $\zeta_0\in\partial\mathbb B^d$ does not lie in the joint
  spectrum of $\pi(X)$ and derive a contradiction. If this were the
  case, then by Lemma~\ref{lem:resolvent-mapping} the element
  $I-\zeta_0^*{\bf X}$ would be invertible modulo compacts, as would $I-\zeta{\bf X}^*_0$, and hence there would exist an open set $\Gamma\subset\partial\mathbb B^d$ containing $\zeta_0$ for which $I-\zeta{\bf X}^*$ was invertible modulo compacts for every $\zeta\in\Gamma$.  In particular, each of the operators $I-\zeta{\bf X}^*$ would be Fredholm and hence have closed range. Thus by Lemma~\ref{lem:closed-range-cpoint}, each $\zeta\in\Gamma$ would be a C-point for $b$, and thus $b$ and $\Gamma$ would satisfy the hypotheses of Theorem~\ref{thm:rudin} for any sequence $r_j\to 1$, but also $\lim_{r\to 1} (D_\zeta f)(r\zeta)$ would exist and be finite for each $\zeta\in\Gamma$. This obviously contradicts (\ref{eqn:derivs-blow-up}).
\end{proof}

\bibliographystyle{plain} 
\bibliography{working-ncac} 

\end{document}